\documentclass[final]{siamart190516}

\usepackage{lipsum}
\usepackage{epstopdf}
\usepackage{amsmath,amsfonts,stmaryrd,amssymb,graphicx,subfigure} 
\usepackage{enumitem}
\usepackage{tikz}
\usetikzlibrary{matrix}
\usepackage[algo2e,ruled]{algorithm2e} 
\usepackage{algorithmic}
\usepackage{multirow}
\ifpdf
\DeclareGraphicsExtensions{.eps,.pdf,.png,.jpg}
\else
\DeclareGraphicsExtensions{.eps}
\fi

\newsiamremark{remark}{Remark}
\newsiamremark{hypothesis}{Hypothesis}
\crefname{hypothesis}{Hypothesis}{Hypotheses}
\newsiamthm{claim}{Claim}
\newsiamremark{example}{Example}

\headers{D-SOS and DC-SOS decompositions for polynomials}{Y.S. Niu et al.}

\title{On Difference-of-SOS and Difference-of-Convex-SOS \\Decompositions for Polynomials\thanks{
		Accepted by the SIAM Journal on Optimization January 30, 2024.
		\funding{The first author is supported by the Natural Science Foundation of China (Grant No: 11601327), the Key Construction National ``$985$" Program of China (Grant No: WF220426001), and the Research Grants Council (Grant No: 15304019).}}}

\author{Yi-Shuai Niu\thanks{Beijing Institute of Mathematical Sciences and Applications, Beijing, 101408, China (\email{niuyishuai@bimsa.cn}, \email{niuyishuai82@hotmail.com}).}
	\and
	Hoai An Le Thi \thanks{Universit\'{e} de Lorraine, LGIPM,
		Metz, F-57000, France and
		Institut Universitaire de France (IUF)
		(\email{hoai-an.le-thi@univ-lorraine.fr}).}
	\and
	Dinh Tao Pham
	\thanks{National Institute of Applied Sciences of Rouen, Laboratory of Mathematics, France (\email{pham@insa-rouen.fr}).}
}

\usepackage{amsopn}

\DeclareMathOperator{\R}{\mathbb{R}}
\DeclareMathOperator{\N}{\mathbb{N}}
\DeclareMathOperator{\Z}{\mathbb{Z}}

\DeclareMathOperator{\calA}{\mathcal{A}}
\DeclareMathOperator{\calN}{\mathcal{N}}
\DeclareMathOperator{\calR}{\mathcal{R}}

\newcommand{\IntEnt}[1]{\left[\!\left[#1\right]\!\right]}
\DeclareMathOperator{\SOS}{\mathcal{SOS}}
\DeclareMathOperator{\PSD}{\mathcal{PSD}}
\DeclareMathOperator{\CSOS}{\mathcal{CSOS}}
\DeclareMathOperator{\D-SOS}{\mathcal{D-SOS}}
\DeclareMathOperator{\DC-SOS}{\mathcal{DC-SOS}}
\DeclareMathOperator{\tr}{trace}
\DeclareMathOperator{\Sp}{\text{Sp}}
\newcommand{\Int}[1]{\textrm{int}(#1)}
\newcommand*{\tran}{^{\mkern-1.5mu\mathsf{T}}}
\newcommand{\ie}{i.e.}
\newcommand{\eg}{e.g.}
\newcommand{\POLYDSOS}{\texttt{POLY2DSOS}}
\newcommand{\POLYDCSOS}{\texttt{POLY2DCSOS}}
\newcommand{\density}{\texttt{den}}

\newenvironment{proc}[2]
{\begin{algorithm}[#1]\floatname{algorithm}{#2:}}
	{\end{algorithm}\addtocounter{algorithm}{-1}}

\ifpdf
\hypersetup{
  pdftitle={D-SOS \& DC-SOS},
  pdfauthor={Y. S. Niu}
}
\fi

\graphicspath{{Figs/}}

\begin{document}

\maketitle

\begin{abstract}
	In this article, we are interested in developing polynomial decomposition techniques based on sums-of-squares (SOS), namely the difference-of-sums-of-squares (D-SOS) and the difference-of-convex-sums-of-squares (DC-SOS). In particular, the DC-SOS decomposition is very useful for difference-of-convex (DC) programming formulation of polynomial optimization problems. First, we introduce the cone of convex-sums-of-squares (CSOS) polynomials and discuss its relationship to the sums-of-squares (SOS) polynomials, the non-negative polynomials and the SOS-convex polynomials. Then, we propose the set of D-SOS and DC-SOS polynomials, and prove that any polynomial can be formulated as D-SOS and DC-SOS. The problem of finding D-SOS and DC-SOS decompositions can be formulated as a semi-definite program and solved for any desired precision in polynomial time using interior point methods. Some algebraic properties of CSOS, D-SOS and DC-SOS are established. Second, we focus on establishing several practical algorithms for exact D-SOS and DC-SOS polynomial decompositions without solving any SDP. The numerical performance of the proposed D-SOS and DC-SOS decomposition algorithms and their parallel versions, tested on a dataset of $1750$ randomly generated polynomials, is reported.  
\end{abstract}



\section{Introduction}\label{sec:intro}
We are interested in establishing polynomial decomposition techniques in forms of difference-of-sums-of-squares (D-SOS) and difference-of-convex-sums-of-squares (DC-SOS). The notation D-SOS is different from the diagonally-dominant-sums-of-squares (DSOS) proposed in  \cite{Paper_Ahmadi2017,ahmadi2019dsos}. In order to avoid potential confusion, we denote D-SOS instead of DSOS. Throughout the paper, the set $\R[x]$ stands for the infinite dimensional vector space of all real-valued polynomials with variable $x\in \R^n$, $n\in \N^*:=\N\setminus\{0\}$ and with coefficients in $\R$. The set $\R_d[x]$ with $d\in \N$ denotes the finite dimensional subspace of $\R[x]$ (with dimension $\binom{n+d}{n}$) for polynomials of degree less than or equal to $d$.

Our initial motivation comes from establishing difference-of-convex (DC) decompositions for polynomials, which is a crucial question in DC programming. Inspired by Hilbert's $17$th problem, there are close relationships between non-negative polynomials and sums-of-squares (SOS) polynomials. A polynomial is not always non-negative (or SOS), but a natural question arises: ``\emph{can we rewrite any polynomial as difference of non-negative (or SOS) polynomials?}" This question motivates us to exploit polynomial DC decomposition techniques involving SOS. 

Consider the polynomial optimization problem defined by
\begin{equation}
	\label{P}
	\min \{f(x): x\in \mathcal{D} \}, \tag{P}
\end{equation}
where $f$ is a finite degree polynomial in $\R[x]$ with $x\in \R^n$, and $\mathcal{D}$ is a nonempty closed convex set of $\R^n$. Problem \eqref{P} is in general nonconvex and has large number of real-world applications such as high-order moment portfolio optimization problems \cite{Paper_Niu2019higher,Paper_Niu2011}, eigenvalue complementarity problems \cite{niu2019improved,Paper_Niu2015,Paper_Niu2013}, Euclidean distance matrix completion problems \cite{bakonyi1995euclidian}, checking copositivity of matrices \cite{Dur2013testing}, Boolean polynomial programs \cite{niu2019discrete}, geometric program \cite{boyd2007tutorial}, sparse PCA \cite{zou2006sparse}, and control of humanoid robots \cite{ahmadi2019dsos}. Problem \eqref{P} can be formulated as a \emph{convex constrained DC program} as
\begin{equation}
	\label{Pdc}
	\min \{f(x) = g(x) - h(x): x\in \mathcal{D} \}, \tag{DC}
\end{equation}
where both $g$ and $h$ are convex polynomials over $\mathcal{D}$. Then, we can solve \eqref{Pdc} using DC programming approaches such as DCA \cite{Paper_Pham1997,Paper_Pham1998,Paper_Pham2005,Paper_Lethi2018} and its variants (e.g., the proximal DCA \cite{souza2016global,wen2018proximal}, the Boosted-DCA  \cite{BDCA_S,Paper_Niu2019higher}, the inertial DCA \cite{de2019inertial}, the ADCA \cite{nhat2018accelerated}, the convex-concave procedure \cite{lipp2016variations,shen2016disciplined}), the linearized proximal methods \cite{souza2016global,pang2017computing}, and the DC bundle methods \cite{joki2017proximal,joki2018double,gaudioso2018minimizing}. 

Although any polynomial function over $\mathcal{D}$ is DC, as any $\mathcal{C}^{2}(\mathcal{D},\R)$ function is DC  \cite{Paper_Hartman1959,Book_Tuy2016}, constructing DC decompositions for general polynomials (especially for high degree multivariate polynomials, with degree larger than 2) is nontrivial. In our knowledge, there are few results in literature on DC decomposition techniques for general polynomials. H. Tuy showed in \cite{Book_Tuy2016} that determining a DC decomposition for a monomial on $\R^n_+$ is easy. Then he proposed to introduce two additional variables $y$ and $z$ in $\R^n_+$ to rewrite $x\in \R^n$ as $y-z$, thus a polynomial in $x$ is rewritten as a polynomial in $(y,z)$, and the latter one is easily expressed in DC. However, this trivial approach doubled the number of variables and is therefore not suitable for polynomials with a large number of variables. A.F. Biosca presented in  \cite{Biosca2001} an interesting DC decomposition to formulate any polynomial as power-sums of linear forms, which is an interesting direction for constructing DC decompositions. However, finding an available set of linear forms with minimal terms to present any given polynomial is not easy and deserves more investigation. Y.S. Niu et al. proposed DC decomposition techniques for polynomials over a compact convex set based on the so-called projective DC decomposition (see e.g., \cite{Paper_Pham1997,hoai2000efficient,Thesis_Niu2010}) in form of $\frac{\rho}{2}\|x\|^2 - (\frac{\rho}{2}\|x\|^2 - f(x))$, where $\rho>0$ is an overestimation of the spectral radius of the Hessian matrix of $f$ over the compact convex set. This technique has been applied in several applications such as the high-order moment portfolio optimization  \cite{Paper_Niu2011} and the eigenvalue complementarity problems  \cite{niu2019improved,Paper_Niu2015,Paper_Niu2013}. However, its drawback is the requirement of a good estimation of $\rho$, which is often difficult to compute. Note that a tighter estimation leads to a better DC decomposition. In \cite{Thesis_Niu2010}, a DC decomposition based on quadratic programming (QP) formulation for polynomials is proposed, then applying DCA for QP formulation requires solving a sequence of convex QPs using efficient QP solvers such as CPLEX \cite{Cplex} and GUROBI \cite{Gurobi}. 

Recently, A.A. Ahmadi et al.  \cite{Paper_Ahmadi2017} investigated DC decompositions for polynomials based on SOS-convexity. They first proved that the convexity of a polynomial is not equivalent to the SOS-convexity  \cite{Paper_Ahmadi2011} and constructed counterexamples in  \cite{Paper_Ahmadi2012}. Then three DC decompositions based on SOS-convexity, namely diagonally-dominant-sums-of-squares-convex, scaled-diagonally-dominant-sums-of-squares-convex and sos-convex decompositions are proposed in \cite{Paper_Ahmadi2017}. Meanwhile, Y.S. Niu et al. proposed in \cite{niu2019improved} a DC decomposition in form of difference-of-convex-sos polynomials for the quadratic eigenvalue complementarity problem, which is equivalent to a particular structured linearly constrained quartic polynomial optimization problem. This DC decomposition yields better numerical results in the quality of the computed solution and the computing time than the DC decompositions in  \cite{Paper_Niu2015,Paper_Niu2013}. A similar idea is also applied to the high-order moment portfolio optimization  \cite{Paper_Niu2019higher} which is a quartic polynomial optimization problem over a standard simplex. 

In this paper, we will generalize the DC decompositions based on SOS used in \cite{niu2019improved} and \cite{Paper_Niu2019higher} to provide some fundamental theorems and practical algorithms for establishing D-SOS and DC-SOS decompositions for general polynomials. Our contributions include:
(i) Establishing convex-sums-of-squares (CSOS), D-SOS and DC-SOS decompositions in \Cref{sec:polydec}. Some algebraic properties and the relationships with SOS, nonnegative polynomials and Convex-SOS polynomials are discussed. We prove that the set of CSOS polynomials is a full-dimensional convex cone, and the set of D-SOS and DC-SOS polynomials are vector spaces and equivalent to $\R[x]$ based on a similar argument as in \cite{Paper_Ahmadi2017}. As an important consequence, any polynomial can be rewritten in forms of D-SOS and DC-SOS, whose decompositions can be computed in polynomial time for any desired precision by solving a semidefinite program (SDP). (ii) Developing several practical algorithms to generate exact D-SOS (\Cref{sec:D-SOS}) and DC-SOS (\Cref{sec:DC-SOS}) decompositions without solving any SDP. (iii) Implementing proposed D-SOS and DC-SOS decomposition algorithms in MATLAB (namely, \texttt{POLY2DSOS} and \texttt{POLY2DCSOS}, codes available at Github\footnote{\url{https://github.com/niuyishuai/POLY2DSOS_DCSOS}}). The parallel codes are also developed for some algorithms based on monomials. Numerical simulations on randomly generated polynomials with different number of variables, polynomial degrees and densities are reported in \Cref{sec:simulations}. 

The significance of studying D-SOS and DC-SOS decompositions is manifold. The DC-SOS decomposition plays a pivotal role by enabling the representation of any polynomial function as a DC function. This, in turn, allows solving the polynomial optimization problems through efficient DC programming algorithms. The D-SOS decomposition, as an extension of the classical polynomial SOS decomposition, can be more easily constructed than the DC-SOS decomposition. Beyond serving as the theoretical foundation for the DC-SOS decomposition, the D-SOS decomposition also paves the way for a novel class of optimization problems - the `D-SOS programming' defined as: 
$$\min_{x\in \mathcal{D}} g(x)-h(x),$$
where both $g$ and $h$ are SOS polynomials, and $\mathcal{D}$ is a non-empty closed convex set in $\R^n$. Despite the difficulty of directly minimizing a D-SOS polynomial, which deserves much attention in our future research, the D-SOS structure empowers us to transform any D-SOS programming problem into a sparse structured DC programming problem as outlined below: Consider that $g(x)=\sum_{i}g_i^2(x)$, $h(x)=\sum_j h_j^2(x)$, where $g_i(x)$ and $h_j(x)$ are some polynomials with $\deg(g_i)\leq \lceil\deg(g)/2\rceil$ and $\deg(h_j)\leq \lceil\deg(h)/2\rceil$. By introducing auxiliary variables $$y_i = g_i(x), \forall i\quad \text{ and } \quad z_j = h_j(x), \forall j,$$ we can represent the D-SOS programming problem as $$\min\{\sum_{i}y_i^2 - \sum_{j}z_j^2: x\in \mathcal{D}, y_i = g_i(x), \forall i, z_j = h_j(x), \forall j\}.$$ Likewise, the D-SOS decomposition of $g_i(x)$ enables the constraint $y_i=g_i(x)$ to be expressed as a DC constraint $y_i = \sum_{k}u_k^2 - \sum_{l} v_l^2$ by adding some equations $u_k = p_k(x)$ and $v_l=q_l(x)$ where $\deg(p_k)\leq \lceil\deg(g_i)/2\rceil$ and $\deg(q_l)\leq \lceil\deg(g_i)/2\rceil$ (analogous to the constraint $z_j=h_j(x)$). We repeat this procedure until the degree of the DC constraint equals $2$. Ultimately, this permits an equivalent representation of the D-SOS programming problem as a quadratic programming problem with a sparse DC objective function $\sum_{i}y_i^2 - \sum_{j}z_j^2$ and accompanied by some sparse quadratic DC constraints of the form $y_i = \sum_{k}u_{k}^2 - \sum_{l} v_{l}^2$. This type of DC constrained DC programs can be tackled using some DC algorithms, see e.g., \cite{le2014dc,lipp2016variations}. This transformation provides a tangible connection between D-SOS decomposition and DC programming formulation. Furthermore, the sparse DC structure can potentially facilitate efficient solutions via DC approaches for some large-scale polynomial optimization problems.

It should be noted that the most attractive advantage of DC-SOS decomposition is the ability to generate good DC decompositions for polynomials (a renowned open question in DC programming \cite{Paper_Lethi2018}), thus reducing the number of iterations in DCA for solving polynomial optimization problems more efficiently. How to solve the convex subproblems required in DCA by exploiting the SOS structure is beyond the scope of this article, which is a very interesting and important research topic in DC programming. Readers may refer to some successful real-world applications using DC-SOS decompositions in quadratic eigenvalue complementarity problems \cite{niu2019improved} and high-order moment portfolio optimization problems \cite{Paper_Niu2019higher}.

\section{Polynomial decompositions based on SOS} \label{sec:polydec}
In this section, we will introduce the CSOS, D-SOS, and DC-SOS polynomials and discuss their algebraic properties, constructibility and complexity.

\subsection{The proper cone of CSOS polynomials}\label{subsec:csos} 

Recall that a polynomial is  \emph{Sums-Of-Squares} (SOS) if it can be rewritten as sums of squares of some polynomials. It is well-known that any SOS polynomial is nonnegative (cf. positive semidefinite - PSD), but a PSD polynomial in $n$ variables and of degree $d$ is SOS if and only if $n = 1$ (univariate polynomials), or $d = 2$ (quadratic
polynomials), or $(n,d) = (2,4)$ (bivariate quartic
polynomials). See more discussions and counterexamples in \cite{Paper_Hilbert1888,Paper_Reznick1996}. Throughout the paper, we will denote $\SOS_{n}$ and $\PSD_{n}$ (resp. $\SOS_{n,2d}$ and $\PSD_{n,2d}$) for SOS and PSD polynomials of $n$ variables (resp. and of degree $2d$). 

\begin{definition}[CSOS polynomial]\label{def:CSOS}
	A polynomial is called \emph{Convex-Sums-Of-Squares} (CSOS) if it is both convex and SOS. The set of CSOS polynomials in $\R[x]$ is denoted by $\CSOS_n$, and the subset of $\CSOS_n$ in $\R_d[x]$ is denoted by $\CSOS_{n,d}$.
\end{definition}

Note that CSOS is different from \emph{SOS-convexity} in  \cite{Paper_Helton2010} defined by:
\begin{definition}[SOS-matrix and SOS-convex polynomial, see \cite{Paper_Helton2010}]\label{def:sos-convex}
	A symmetric polynomial matrix $P(x)\in \R[x]^{m\times m}$ is called an \emph{SOS-matrix} if there exists a polynomial matrix $M(x)\in \R[x]^{s\times m}$ such that
	$$P(x) = M\tran(x)  M(x).\footnote{$M\tran$ stands for the transpose of the matrix $M$.}$$
	A polynomial is called \emph{SOS-convex} if its Hessian matrix is an SOS-matrix.
\end{definition}

SOS-convexity is a sufficient condition for the convexity of a polynomial since its Hessian matrix is positive semidefinite. In general, deciding the convexity of a polynomial appeared in the list of seven open questions in complexity theory for numerical optimization in 1992  \cite{Paper_Pardalos1992}. In recent years, it has been proved in  \cite{Paper_Ahmadi2013} that this problem is strongly NP-hard for polynomials of degree four. However, SOS-convexity can be certified in polynomial time based on the next lemma:
\begin{lemma}(See  \cite{Paper_Kojima2014})\label{lemma:sosmatrix}
	A polynomial matrix $P(x)\in \R[x]^{m\times m}$ is an SOS-matrix if and
	only if $y\tran   P(x)   y$ is an SOS polynomial in $\R[(x,y)]$.
\end{lemma}

\paragraph{\textbf{Algebraic properties of CSOS polynomials}}
The sets of SOS-convex polynomials and PSD convex polynomials are proper cones as proved in \cite{Paper_Ahmadi2017}. We will now show a similar result in CSOS polynomials. 
\begin{proposition}\label{thm:1-1}
	$\CSOS_{n}$ (resp. $\CSOS_{n,2d}$) is a proper cone in $\R[x]$ (resp. $\R_{2d}[x]$). 
\end{proposition}
\begin{proof}
	The set $\CSOS_{n}$ (resp. $\CSOS_{n,2d}$) is pointed since it is a subset of the pointed set $\SOS_n$ (resp. $\SOS_{n,2d}$). The closedness and convexity are immediate consequences of the recession cone theorem  \cite{Book_Bertsekas2009} since $\CSOS_{n}$ (resp. $\CSOS_{n,2d}$) is a recession cone of the proper cone $\SOS_n$ (resp. $\SOS_{n,2d}$), and any recession cone of nonempty closed convex set is closed and convex. To show that $\CSOS_{n}$ (resp. $\CSOS_{n,2d}$) is full-dimensional in $\R[x]$ (resp. $\R_{2d}[x]$), since the set of SOS-convex polynomials is a proper cone, the nonnegative SOS-convex polynomials is also a proper cone which is a subset of CSOS polynomials  \cite{Paper_Helton2010}. Finally, we conclude that $\CSOS_{n}$ (resp. $\CSOS_{n,2d}$) is a proper cone in $\CSOS_{n}$ (resp. $\CSOS_{n,2d}$). 
\end{proof}

Some properties of CSOS polynomials are summarized as follows whose proofs are trivial and will be omitted.
\begin{proposition}\label{prof:csos}
	The set $\CSOS_n$ (resp. $\CSOS_{n,2d}$) satisfies:
	\begin{enumerate}
		\item[(i)] Let $p\in \CSOS_n$ (resp. $\CSOS_{n,2d}$). Then $|p|:=\sup\{p,-p\}$, $p^+:=\sup\{p,0\}$ and $p^-:=\inf\{p,0\}$ are polynomials in $\CSOS_n$ (resp. $\CSOS_{n,2d}$).
		\item[(ii)] The set $\CSOS_n$ (resp. $\CSOS_{n,2d}$) is not closed for multiplication, subtraction, division, $\sup$ and $\inf$.
	\end{enumerate}
\end{proposition}

Note that \cref{prof:csos} is also valid for PSD convex polynomials, but not for SOS-convex polynomials. Because $|p|$, $p^+$ and $p^-$ may not be polynomials due to the absence of  nonnegativity for SOS-convex polynomials.   

\paragraph{\textbf{Complexity of CSOS certification}}
It is known that deciding the nonnegativity of a polynomial of degree $\geq 4$ is NP-hard  \cite{Paper_Murty1987} (since deciding the matrix copositivity is NP-complete). However, decomposing a polynomial into SOS is equivalent to solving a semidefinite program (SDP)\footnote{An SDP is a convex optimization problem as $\min\{f(X): \mathcal{A}(X)=0, X\succeq 0\}$ where $X\succeq 0$ means that $X$ is a real symmetric positive semidefinite matrix, $f$ is a linear function, and $\mathcal{A}$ is an affine map of $X$.}, which can be solved, under certain regularity conditions (e.g, Slater's condition), to any desired precision in polynomial time using interior point methods  \cite{Paper_Parrilo2003}. One can use software such as YALMIP  \cite{Yalmip} and SOSTOOLS  \cite{SOSTOOLS} to check if a polynomial is SOS, and to find its SOS decomposition.

Now it is easy to show that the complexity for checking a polynomial of even degree ($\geq 4$) to be CSOS is NP-hard, because a convincing example in  \cite{Paper_Ahmadi2013} shows that determining the convexity of a quartic form (a CSOS polynomial) is NP-hard.	For this reason, in practice, we are particularly interested in polynomials that are both CSOS and SOS-convex (i.e., PSD SOS-convex), since they can be checked in polynomial time and hold all properties of CSOS polynomials. Note that CSOS and SOS-convex are not equivalent whose relationship will be discussed below.

\paragraph{\textbf{Relationships (CSOS, SOS-convexity and PSD convexity)}} 
\begin{proposition}\label{prop:2.9}
	A CSOS or PSD convex polynomial may not be SOS-convex. Conversely, an SOS-convex polynomial can be neither CSOS nor PSD convex.
\end{proposition}
\begin{proof}
	It is shown in  \cite{Paper_Ahmadi2012} that the next polynomial
	\begin{equation}\label{eq:aaa'spoly}
		\begin{array}{rl}
			p(x)=&32x_1^8+118x_1^6x_2^2+40x_1^6x_3^2+25x_1^4x_2^4-43x_1^4x_2^2x_3^2\\
			&-35x_1^4x_3^4+3x_1^2x_2^4x_3^2-16x_1^2x_2^2x_3^4+24x_1^2x_3^6+16x_2^8\\
			&+44x_2^6x_3^2+70x_2^4x_3^4+60x_2^2x_3^6+30x_3^8
		\end{array}
	\end{equation}
	is PSD, SOS and convex, i.e., both CSOS and PSD convex, but it is not SOS-convex. Conversely, the SOS-convexity does not imply nonnegativity, \eg, the univariate quadratic polynomial $p(x) = x^2-x-1$ is SOS-convex since its Hessian matrix $\nabla^2 p(x) = 2$ is an SOS-matrix, but $p(x)$ is neither SOS nor PSD. 
\end{proof}

The gap between the SOS-convexity and the convexity of polynomials is characterized in
 \cite{Paper_Ahmadi2011} as a convex polynomial in $n$ variables and of degree $d$ is SOS-convex if and only if $n = 1$ or $d = 2$ or $(n,d)=(2,4)$. It is interesting to observe that PSD is equivalent to SOS exactly in dimensions and degrees where convexity is equivalent to SOS-convexity. Although the question remains open whether there is a deep connection between them.

An interesting question is: \emph{Can we rewrite any SOS-convex polynomial as a CSOS polynomial plus a constant?} This question arises from the observation that the SOS-convex polynomial $x^2-x-1$ can be rewritten as the CSOS polynomial $(x-\frac{1}{2})^2$ plus $- \frac{5}{4}$. Unfortunately, the answer to this question is NO, for example, the polynomial $p(x,y) = x^2 - y +1$ is SOS-convex since its Hessian matrix $\nabla^2 p(x,y) = \begin{bmatrix}
	2 & 0\\
	0 & 0
\end{bmatrix} = \begin{bmatrix}
	\sqrt{2}\\
	0
\end{bmatrix}   \begin{bmatrix}
	\sqrt{2} & 0
\end{bmatrix}$ is an SOS-matrix. However, $p(0,y)\to -\infty$ as $y\to +\infty$ implies that there is no constant $c\in \R$ such that $p(x,y)+c$ is PSD. Hence, $p(x,y)$ cannot be rewritten as a CSOS polynomial plus a constant. This convincing example shows that we cannot regard CSOS as a special case or an extension of SOS-convex. 

Regarding PSD convex versus CSOS. We know that the PSD convex polynomial form is not equivalent to the CSOS polynomial form. Blekherman \cite{Paper_Blekherman2009} showed that there exist convex forms that are not SOS in a nonconstructive way. Saunderson recently provided a specific example in \cite{saunderson2022convex} that a convex form of degree 4 in 272 variables is not SOS. Hence, we have no equivalence between CSOS polynomials and PSD convex polynomials. The next proposition summarizes some known facts.
\begin{proposition}\label{prop:CSOSvsPSDconvex}
	The set of CSOS polynomials is not equivalent to the set of PSD convex polynomials and 
	\begin{enumerate} 
		\item[(i)] Any CSOS polynomial is PSD convex, the converse may not be true.
		\item[(ii)] Any PSD SOS-convex polynomial is CSOS, the converse may not be true.
	\end{enumerate}	
\end{proposition}
\begin{proof}
	(i) By the definition of CSOS, it is clear that CSOS implies PSD convex. Conversely, the example of Saunderson in \cite{saunderson2022convex} serves as a counterexample that a PSD convex polynomial may not be CSOS even to polynomial forms. Note that the polynomial given in \cref{eq:aaa'spoly} cannot serve as a counterexample since it is indeed CSOS.\\ 
	(ii) It is shown in  \cite{Paper_Helton2010} that a nonnegative SOS-convex polynomial is both convex and SOS (\ie, CSOS). Conversely, \Cref{prop:2.9} shows that CSOS may not be SOS-convex.
\end{proof}	

The relationships among the sets of convex, PSD, SOS, CSOS, SOS-convex and PSD convex polynomials are summarized below:
\begin{center}
	\begin{tikzpicture}
		\matrix[matrix of nodes] {
			PSD SOS-Convex & & $\subset$ & & CSOS & $\subset$ & SOS\\
			$\cap$ & & & & $\cap$ & & $\cap$\\
			SOS-Convex & $\subset$ & Convex & $\supset$ & PSD convex & $\subset$ & PSD\\
		};
	\end{tikzpicture}
\end{center}
Note that PSD SOS-convex are identical to CSOS SOS-convex and SOS SOS-convex.

\subsection{The vector spaces of D-SOS and DC-SOS polynomials}
\begin{definition}[D-SOS and DC-SOS polynomial]\label{def:D-SOS}
	If a polynomial $p$ can be written as 
	$p = s_1 - s_2$
	where $s_1$ and $s_2$  are SOS (resp. CSOS) polynomials, then $p$ is called \emph{Difference-of-Sums-Of-Squares} (D-SOS) (resp. \emph{Difference-of-Convex-Sums-Of-Squares} (DC-SOS)). 
	The polynomials $s_1$ and $s_2$ are termed as  \emph{D-SOS (resp. DC-SOS) components} of $p$. The set of D-SOS (resp. DC-SOS) polynomials of $\R[x]$ is denoted by $\D-SOS_n$ (resp. $\DC-SOS_{n}$), and the subset of $\D-SOS_n$ (resp. $\DC-SOS_{n}$) in $\R_d[x]$ is denoted by $\D-SOS_{n,d}$ (resp. $\DC-SOS_{n,d}$).
\end{definition}

Note that the D-SOS and DC-SOS polynomials could be of either even or odd degrees, whereas both SOS and CSOS polynomials are only of even degrees.   

\paragraph{\textbf{Algebraic properties of D-SOS and DC-SOS}}
\begin{theorem}\label{thm:1-2}
	$\D-SOS_n$ and $\DC-SOS_n$ (resp. $\D-SOS_{n,d}$ and $\DC-SOS_{n,d}$) are vector subspaces of $\R[x]$ (resp. $\R_d[x]$).
\end{theorem}
The proof is trivial and thus omitted. Note that $\D-SOS_{n}$ (resp. $\DC-SOS_{n}$) is a vector space spanned by $\SOS_{n}$ (resp. $\CSOS_{n}$) as $\D-SOS_{n} = \SOS_{n} - \SOS_{n}$ (resp. $\DC-SOS_{n} = \CSOS_{n} - \CSOS_{n}$). Moreover, $\DC-SOS_{n}\subset \D-SOS_{n}$ since $\CSOS_{n}\subset \SOS_{n}$. We even have a stronger result that $\DC-SOS_{n} = \D-SOS_{n}$ as proved later in \Cref{thm:universalDSOS&DCSOS}.

The next proposition summarizes some usual operations on $\D-SOS_n$ and $\DC-SOS_n$.
\begin{proposition}\label{prop:CVofGP-DSOS-2}
    The set $\D-SOS_n$ (resp. $\DC-SOS_n$) has the property that for any pair $(p,q) \in \D-SOS_n^2$ (resp. $\DC-SOS_n^2$), it holds that $p \times q \in \D-SOS_n$ (resp. $\DC-SOS_n$). However, $|p|$, $p^+$, $p^-$, $\sup\{p,q\}$, and $\inf\{p,q\}$ do not belong to $\D-SOS_{n}$ (resp. $\DC-SOS_{n}$).
\end{proposition}
\begin{proof} The closedness of multiplication is obtained as follows: for any $p$ and $q$ in $\D-SOS_n$ (resp. $\DC-SOS_n$) with decompositions $$p = \sum_{i=1}^{m}p_i^2 - \sum_{j=1}^{k}\hat{p}_j^2 \quad \text{and}\quad q = \sum_{i'=1}^{m'}q_{i'}^2 - \sum_{j'=1}^{k'}\hat{q}_{j'}^2.$$
	Then, a D-SOS decomposition of $p\times q$ is  
	\begin{equation}\label{eq:prodD-SOS}
		\footnotesize\underbrace{\left(\sum_{i=1}^{m}\sum_{i'=1}^{m'} (p_i\times q_{i'})^2 + \sum_{j=1}^{k}\sum_{j'=1}^{k'}(\hat{p}_j\times \hat{q}_{j'})^2 \right)}_{\in \SOS_n} 
			- \underbrace{\left(\sum_{i=1}^{m}\sum_{j'=1}^{k'} (p_i\times\hat{q}_{j'})^2 + \sum_{j=1}^{k}\sum_{i'=1}^{m'}(\hat{p}_j\times q_{i'})^2 \right)}_{\in \SOS_n},
	\end{equation}	
	and a DC-SOS decomposition of $p\times q$ is  
	\begin{equation}\label{eq:prodDC-SOS}
		\footnotesize\underbrace{\frac{1}{2}\left[ \left( \sum_{i=1}^{m}p_i^2 + \sum_{i'=1}^{m'}q_{i'}^2\right)^2 + \left(\sum_{j=1}^{k}\hat{p}_j^2 + \sum_{j'=1}^{k'}\hat{q}_{j'}^2 \right)^2 \right]}_{\in \CSOS_n} - \underbrace{\frac{1}{2}\left[ \left( \sum_{i=1}^{m}p_i^2 + \sum_{j'=1}^{k'}\hat{q}_{j'}^2 \right)^2 + \left(\sum_{j=1}^{k}\hat{p}_j^2 + \sum_{i'=1}^{m'}q_{i'}^2\right)^2 \right]}_{\in \CSOS_n}. 
	\end{equation} 
	The operations $|p|$, $p^+$, $p^-$, $\sup\{p,q\}$ and $\inf\{p,q\}$ are not closed since they may not be polynomials. For instance, let $p(x)=x^2-x^4$ and $q(x)=x^4$ with $x\in \R$, then the polynomials $p$ and $q$ are both D-SOS and DC-SOS, but $|p|$, $p^+$, $p^-$, $\sup\{p,q\}$ and $\inf\{p,q\}$ are not polynomials. 
\end{proof}

Note that for DC-SOS polynomials $p$ and $q$, although $|p|$, $p^+$, $p^-$, $\sup\{p,q\}$ and $\inf\{p,q\}$ may not be polynomials, but they are still DC functions. 

An important question is the \emph{representability} of any polynomial as D-SOS or DC-SOS. As the main results of D-SOS and DC-SOS decompositions, we will prove that this assertion is true as in \Cref{thm:universalDSOS&DCSOS}, and such decompositions can be constructed by solving an SDP as discussed later in \Cref{thm:complexityforD-SOSandDC-SOS}.

\begin{theorem}\label{thm:universalDSOS&DCSOS}
	Let $n\in \N^*$ and $x\in \R^n$. Then
	\begin{itemize}
		\item[(i)] $\R[x]=\D-SOS_{n} = \DC-SOS_{n}$.
		\item[(ii)] For any D-SOS (resp. DC-SOS) components $s_1$ and $s_2$ of polynomial $p\in \R[x]$, we have $\max\{\deg(s_1),\deg(s_2)\}\geq 2\lceil\frac{\deg(p)}{2} \rceil$.
		\item[(iii)] For any $p\in \R[x]$, there exist D-SOS (resp. DC-SOS) components $s_1$ and $s_2$ such that $\max\{\deg(s_1),\deg(s_2)\}= 2\lceil\frac{\deg(p)}{2} \rceil$.
	\end{itemize}
\end{theorem}
\begin{proof}
	(i) We first prove that $\R_{2d}[x]\subset \D-SOS_{n,2d}$ and $\R_{2d}[x]\subset \DC-SOS_{n,2d}$. The reverse inclusions are obvious. Let $p\in \R_{2d}[x]$ and assume $p\notin \D-SOS_{n,2d}$ (resp. $p\notin\DC-SOS_{n,2d}$). Then, based on \Cref{thm:1-1}, the set $\SOS_{n,2d}$ (resp. $\CSOS_{n,2d}$) is a proper cone, so $\Int{\SOS_{n,2d}}$\footnote{$\Int{A}$ stands for the interior of the set $A$.} (resp. $\Int{\CSOS_{n,2d}}$) is nonempty. Therefore, for any $p_1\in \Int{\SOS_{n,2d}}$ (resp. $\Int{\CSOS_{n,2d}}$) with $p_1\neq p$, there exists an element $p_2$ in the segment $[p,p_1]$ such that $p_2\in \SOS_{n,2d}$ (resp.  $\CSOS_{n,2d}$). Thus, $\exists\lambda \in (0,1)$ such that $p_2 = \lambda\times p + (1-\lambda)\times p_1,$
	which is rewritten as  
	$$p = \frac{1}{\lambda}\times p_2 - \frac{(1-\lambda)}{\lambda}\times p_1$$
	with $\frac{1}{\lambda}\times p_2$ and $\frac{(1-\lambda)}{\lambda}\times p_1$ in $\SOS_{n,2d}$ (resp. $\CSOS_{n,2d}$). Therefore, $p$ can be presented as difference of two vectors in $\SOS_{n,2d}$ (resp. $\CSOS_{n,2d}$) implying that $p\in \D-SOS_{n,2d}$ (resp. $p\in\DC-SOS_{n,2d}$). This leads to a contradiction! Hence \begin{equation}\label{eq:propi}\R_{2d}[x] = \DC-SOS_{n,2d} = \D-SOS_{n,2d}.\end{equation}
	In addition of \begin{equation}\label{eq:propii}
\R_{2d+1}[x]\subset\R_{2d+2}[x] = \DC-SOS_{n,2d+2},\end{equation} we conclude immediately from \eqref{eq:propi} and \eqref{eq:propii} that $$\R[x]=\D-SOS_{n} = \DC-SOS_{n}.$$
	(ii) It follows from $s_1-s_2 = p$ that $\max\{\deg(s_1),\deg(s_2)\}\geq \deg(p)$. Moreover, $s_1$ and $s_2$ are of even degrees implies that $\max\{\deg(s_1),\deg(s_2)\}$ is even and $\geq 2\lceil\frac{\deg(p)}{2} \rceil$.\\
	(iii) For any $p\in \R[x]$ of even or odd degree, we get from \eqref{eq:propi} and \eqref{eq:propii} that there exist D-SOS (resp. DC-SOS) components $s_1$ and $s_2$ with $\max\{\deg(s_1),\deg(s_2)\}\leq 2\lceil\frac{\deg(p)}{2} \rceil$. Combining with (ii), we have $\max\{\deg(s_1),\deg(s_2)\}= 2\lceil\frac{\deg(p)}{2} \rceil$. 
\end{proof}

\paragraph{\textbf{Complexity of D-SOS and DC-SOS decompositions}}
\begin{theorem}\label{thm:complexityforD-SOSandDC-SOS}
	A D-SOS (resp. DC-SOS) decomposition of $p\in \R_d[x]$ can be constructed for any desired precision in polynomial time by solving an SDP.
\end{theorem}
\begin{proof} We first prove the D-SOS decomposition, then extend the result to the DC-SOS decomposition.\\ 
$\rhd$ For the D-SOS decomposition: Let $b(x)=(1,x_1,\ldots,x_n,x_1x_2,\ldots,x_n^{\lceil\frac{d}{2}\rceil})\tran$ be a vector of all monomials in variable $x\in \R^n$ and of degree up to $\lceil\frac{d}{2}\rceil$ (namely, the full-basis of $\R_{\lceil\frac{d}{2}\rceil}[x]$). Then, for any $p\in \R_d[x]$, there exist two real symmetric positive semidefinite matrices $Q_1$ and $Q_2$ such that 
	\begin{equation}\label{eq:gramianform}
		p=(b\tran  Q_1  b)-(b\tran  Q_2  b).
	\end{equation}
Thus, all desired D-SOS decompositions of $p$ are included in the set 
	$$S(b)=\{ (Q_1,Q_2): \underbrace{p=b\tran  (Q_1-Q_2)  b}_{\text{linear equations}}, Q_1\succeq 0, Q_2\succeq 0 \},$$
	where the constraint $p=b\tran  (Q_1-Q_2)  b$ is affine in $(Q_1,Q_2)$. Hence, $S(b)$ is in fact an SDP. Moreover, for any solution $(Q_1,Q_2)\in S(b)$ and for any $a>0$, we have $(Q_1+aI,Q_2+aI)\in S(b)$ with $Q_1+aI\succ 0, Q_2+aI\succ 0$. Thus, the Slater condition holds and $S(b)$ can be solved in polynomial time using interior point methods. \\
$\rhd$ Similarly, finding a DC-SOS decomposition amounts to searching in $S(b)$ a decomposition in form of \cref{eq:gramianform} such that $b\tran  Q_1  b$ and $b\tran  Q_2  b$ are both convex. By the fact that checking the SOS-convexity can be done in polynomial time, so that we can find a difference of SOS-convex SOS decomposition (a special DC-SOS decomposition) for $p$ by solving the feasibility problem:
	$$\widehat{S}(b)=\{ (Q_1,Q_2)\in S(b): \nabla^2 (b\tran  Q_1  b) \text{ and } \nabla^2(b\tran  Q_2  b) \text{ are SOS-matrices}\}.$$
	Based on \Cref{lemma:sosmatrix}, checking the polynomial matrix $\nabla^2 (b\tran  Q_i  b), i=1,2$ being SOS-matrix is equivalent to an SDP, so that $\widehat{S}(b)$ is again an SDP, which can be solved for any desired precision in polynomial time using interior point methods.
\end{proof}

Note that the difference of diagonally-dominant-sos-convex polynomial decomposition introduced in  \cite{Paper_Ahmadi2017} is in fact a special case of D-SOS and DC-SOS decompositions. 

Concerning the numerical solution of SDP, there are several SDP solvers such as MOSEK (by MOSEK ApS) \cite{Mosek}, SeDuMi (by Jos F. Sturm)  \cite{SeDuMi}, SDPT3 (by Kim-Chuan Toh, Michael J. Todd, and Reha H. Tutuncu)  \cite{SDPT3}, CSDP (by Helmberg, Rendl, Vanderbei, and Wolkowicz)  \cite{CSDP}, SDPA (by Masakazu Kojima, Mituhiro Fukuda et al.)  \cite{SDPA} and DSDP(by Steve Benson, Yinyu Ye, and Xiong Zhang)  \cite{DSDP}. Although SDP can be solved for any desired precision in polynomial time, but numerically solving large-scale and dense SDP is very expensive, e.g., the SDP resulted with the full-basis $b(x)$ is indeed in this case. So it is strongly not suggested to solve an SDP for constructing D-SOS and DC-SOS decompositions in practice. Moreover, the numerical solution of an SDP can only provide approximate D-SOS and DC-SOS decompositions, which could be unstable for some ill-conditioned polynomials. 

In the rest of the paper, we will focus on establishing some practical algorithms for constructing exact D-SOS and DC-SOS polynomial decompositions without solving any SDP, which are more efficient and stable in practice.  

\section{Spectral D-SOS decomposition techniques}\label{sec:D-SOS}
In this section, we will establish a D-SOS decomposition technique, namely \emph{spectral D-SOS decomposition}, based on the well-known spectral decomposition of real symmetric matrix. First, a general algorithm of the spectral D-SOS decomposition is proposed in \Cref{subsec:GS-DSOS}; Then, two variants based on \emph{direct basis} and \emph{minimal basis} are presented respectively in \Cref{subsec:DB-DSOS,subsec:MBS-DSOS}. 

\subsection{General spectral D-SOS decomposition}\label{subsec:GS-DSOS} Let us start by introducing the definition of the \emph{valid basis} to represent a polynomial using a set of monomials and a real symmetric matrix.

\begin{definition}[Valid basis]\label{def:validbasis}
	A \emph{valid basis} of $p\in \R[x]$ is a set of monomials, denoted by $b$ (in a vector form), such that there exists a real symmetric matrix $Q$ verifying $p=b\tran  Q  b.$
\end{definition}

For example, the set of all monomials in $\R_{\lceil\frac{d}{2}\rceil}[x]$ (as in \Cref{thm:complexityforD-SOSandDC-SOS}) is a valid basis for any polynomials in $\R_d[x]$. Clearly, a valid basis for a polynomial may not be unique, e.g., the polynomial $x_1+x_1^2x_2$ has many valid bases such as
$$b(x)=\begin{bmatrix}
	1\\
	x_1\\
	x_1^2x_2
\end{bmatrix}, Q = \begin{bmatrix}
	0 & \frac{1}{2} & \frac{1}{2} \\
	\frac{1}{2} & 0 & 0\\
	\frac{1}{2} & 0 & 0\\
\end{bmatrix} \text{ and }
b(x)=\begin{bmatrix}
	1\\
	x_1\\
	x_1x_2
\end{bmatrix}, Q = \begin{bmatrix}
	0 & \frac{1}{2} & 0 \\
	\frac{1}{2} & 0 & \frac{1}{2}\\
	0 & \frac{1}{2} & 0\\
\end{bmatrix}.$$

For any valid basis $b$, the associated symmetric matrix $Q$ can be computed by solving a linear system extracted from the equation
\begin{equation}\label{eq:computeQ}
	p = b\tran  Q   b,
\end{equation}
which provides a method to check a given set of monomials $b$ to be a valid basis for a polynomial $p$. If the linear system associated with \eqref{eq:computeQ} has no solution, then $b$ is not valid. Otherwise, for any solution $\widehat{Q}$ (maybe non-symmetric), the symmetric matrix $Q=\frac{1}{2}(\widehat{Q}^{\top}+\widehat{Q})$ is also a solution. 

\begin{proposition}[Spectral D-SOS Decomposition]
	\label{thm:spectralD-SOSdecomposition}
	Let $p(x)$ be a polynomial in $\R[x]$, $b(x)$ be a valid basis of $p(x)$ with associated symmetric matrix $Q$, $r_b$ be the length of $b(x)$, $\Sp(Q) =\{\lambda_1,\ldots,\lambda_{r_b}\}$ be the set of eigenvalues of $Q$, and $\mathcal{K}=\{k\in \IntEnt{1,r_b}\footnote{$\IntEnt{a,b}$ stands for the set of integers included in the interval $[a,b]$.}:0\neq \lambda_k\in \Sp(Q)\}$ be the index set of nonzero eigenvalues of $Q$. Suppose that $Q$ has a spectral decomposition as $Q=P  \Lambda   P^{\top}$ where $P$ is an orthogonal matrix and $\Lambda$ is a diagonal matrix. Then taking $y(x)=P\tran  b(x)$, we have a D-SOS decomposition for $p(x)$ (namely, \emph{spectral D-SOS decomposition}) defined by: 
	\begin{equation}\label{eq:specD-SOS}
		p(x) = \sum_{i\in \mathcal{K}}\lambda_i y_i^2(x).
	\end{equation}
\end{proposition}
\begin{proof}
	By the symmetry of $Q$, there exists a spectral decomposition of $Q$ as $Q = P  \Lambda   P\tran.$ Then
	$p(x) = b\tran(x)   Q  b(x) = b\tran(x)   P   \Lambda   P\tran   b(x).$
	Let $y(x)=P\tran  b(x)$. Then 
	$p(x) = \sum_{j}\lambda_j y_j^2(x) = \sum_{i\in \mathcal{K}}\lambda_i y_i^2(x).$
\end{proof}

Based on \Cref{thm:spectralD-SOSdecomposition}, the \emph{General Spectral D-SOS Decomposition Algorithm} (cf. GS-DSOS) is established as in \Cref{alg:GS-DSOS}.
\begin{algorithm}[H]
	\caption{General Spectral D-SOS Decomposition (GS-DSOS)}
	\label{alg:GS-DSOS}
	\KwIn{Polynomial $p(x)$.} 
	\KwOut{D-SOS decomposition.}
	\BlankLine
	\textbf{Step 1:} Find a valid basis $b(x)$ and compute its associated matrix $Q$ for $p(x)$.
	
	\textbf{Step 2:} Compute spectral decomposition of $Q$ to get $P$ and $\Lambda$ verifying
	$Q=P  \Lambda  P\tran.$
	
	\textbf{Step 3:} Let $r_b$ be the length of basis $b(x)$, compute $y(x)= P\tran  b(x)$ and $\mathcal{K}=\{k\in \IntEnt{1,r_b}:0\neq \lambda_k\in\Sp(Q)\}$. Then a D-SOS decomposition of $p(x)$ is given by \eqref{eq:specD-SOS}.
\end{algorithm}
\begin{proposition}\label{prop:degreeofGS-DSOS}
	Let $p\in \R[x]$ and $b$ be a valid basis of $p$. Then the maximal degree of D-SOS components generated by \Cref{alg:GS-DSOS} is not greater than $2\deg(b)$.
\end{proposition}
\begin{proof}
	Based on \cref{eq:specD-SOS}, the maximal degree of D-SOS components is not greater than $\displaystyle\max_{i\in \mathcal{K}}\deg(y_i^2(x))\leq 2\deg(y(x))$, then it follows from $\deg(y(x))=\deg(b(x))$ (since $y(x)=P\tran  b(x)$) that the maximal degree of D-SOS components is not greater than $2\deg(b)$. 
\end{proof}

Note that different valid bases leads to different D-SOS decompositions. An important question is: \emph{How to find a valid basis?} We introduce in the following two sections two valid bases, namely \emph{direct basis} and \emph{minimal basis}, and the corresponding spectral D-SOS decompositions.

\subsection{Direct basis spectral D-SOS decomposition}\label{subsec:DB-DSOS}
\begin{definition}[Direct basis]\label{def:direct_basis}
	For a polynomial in form of $p(x) = c_{\alpha^0} + c_{\alpha^1}x^{\alpha^1} + \ldots + c_{\alpha^r}x^{\alpha^r}$ where $(x^{\alpha^i})_{i\in\IntEnt{0,r}}$ are monomials and $(c_{\alpha^i})_{i\in \IntEnt{0,r}}$ are associated coefficients, the vector $b(x)=[1,x^{\alpha^1}, \ldots, x^{\alpha^r}]^{\top}$ is a valid basis (namely \emph{direct basis}) with associated matrix $Q$ defined by
	\begin{equation}
		\label{eq:directbasisQ}
		Q=\left[\begin{array}{c|ccc}
			c_{\alpha^0} & \frac{1}{2}c_{\alpha^1} & \cdots & \frac{1}{2}c_{\alpha^r}\\
			\hline
			\frac{1}{2}c_{\alpha^1} \\
			\vdots & & 0\\
			\frac{1}{2}c_{\alpha^r} 
		\end{array} \right].
	\end{equation}
\end{definition}

Note that the degree of direct basis is equal to the degree of $p$, and the constant monomial $1$ must be included in the direct basis even for a polynomial without constant part (i.e., $c_{\alpha^0}=0$). E.g., the direct basis of $5+x_1x_2$ is $[1, x_1x_2]^{\top}$, and the direct basis of $x_1+x_1^2$ is $[1,x_1,x_1^2]^{\top}$. 

\begin{proposition}\label{thm:eigensofQ}
	Let $Q$ be the matrix defined in \eqref{eq:directbasisQ}. Then $$\Sp(Q) = \{ 0, \lambda^+, \lambda^- \},$$
	with
	\begin{equation}\label{eq:eigenvaluesofdirectbasis}
		\lambda^{+}=\frac{1}{2}\left(c_{\alpha^0}+\sqrt{\sum_{i=0}^{r}c_{\alpha^i}^2}\right)\geq 0 \quad \text{and} \quad \lambda^{-}=\frac{1}{2}\left(c_{\alpha^0}-\sqrt{\sum_{i=0}^{r}c_{\alpha^i}^2}\right)\leq 0.\end{equation}
	If $\lambda^+>0$, then the associated eigenvector is
	\begin{equation}
		\label{eq:vplus}
		v^+ = \left[\displaystyle\sum_{i=1}^{r}c_{\alpha^i}^2,c_{\alpha^1}\left(\sqrt{\displaystyle\sum_{i=0}^{r}c_{\alpha^i}^2 }-c_{\alpha^0}\right),\ldots, c_{\alpha^r}\left(\sqrt{\displaystyle\sum_{i=0}^{r}c_{\alpha^i}^2 }-c_{\alpha^0}\right)\right]\tran.
	\end{equation}
	If $\lambda^-<0$ then the associated eigenvector is 
	\begin{equation}
		\label{eq:vminus}
		v^- = \left[\displaystyle\sum_{i=1}^{r}c_{\alpha^i}^2,-c_{\alpha^1}\left(\sqrt{\displaystyle\sum_{i=0}^{r}c_{\alpha^i}^2 }+c_{\alpha^0}\right),\ldots, -c_{\alpha^r}\left(\sqrt{\displaystyle\sum_{i=0}^{r}c_{\alpha^i}^2 }+c_{\alpha^0}\right)\right]\tran.
	\end{equation}
\end{proposition}
\begin{proof}
	The characteristic polynomial of $Q$ is 
	$$\chi_{_{Q}}(\lambda)=(-1)^{r+1}\lambda^{r-1}(\lambda^2-c_{\alpha^0}\lambda -\frac{1}{4}\sum_{i=1}^{r}c_{\alpha^i}^2),$$ 
	which implies that  $$\Sp(Q) =\{ 0, \lambda^+, \lambda^- \}.$$
	It follows from $|c_{\alpha^0}|\leq \sqrt{\sum_{i=0}^{r}c_{\alpha^i}^2}$ that $\lambda^+=(c_{\alpha^0}+\sqrt{\sum_{i=0}^{r}c_{\alpha^i}^2})/2 \geq 0$ and $\lambda^{-}=(c_{\alpha^0}-\sqrt{\sum_{i=0}^{r}c_{\alpha^i}^2})/2\leq 0$. If $\lambda^+>0$ (or $\lambda^{-}<0$), then it is a simple eigenvalue, and the dimension of its eigenspace is equal to one. Then, it is obvious to check that $v^+$ (or $v^-$) is the associate eigenvector.
\end{proof}

Based on \Cref{thm:eigensofQ}, the spectral D-SOS decomposition in \Cref{alg:GS-DSOS} is simplified as:
\begin{equation}
	\label{eq:D-SOSfordirectbasis}
	p(x)=\sum_{i\in \mathcal{K}}\lambda_i y_i^2(x) = \lambda^+ \frac{(b\tran(x)  v^+)^2}{\|v^+\|^2} + \lambda^- \frac{(b\tran(x)  v^-)^2}{\|v^-\|^2},
\end{equation}
and the corresponding algorithm, namely \emph{Direct Basis Spectral D-SOS Decomposition Algorithm} (cf. DBS-DSOS), is summarized in \Cref{alg:DBS-DSOS}.
\begin{algorithm}[H]
	\caption{Direct Basis Spectral D-SOS Decomposition (DBS-DSOS)}
	\label{alg:DBS-DSOS}
	\KwIn{Polynomial $p(x)$.} 
	\KwOut{D-SOS decomposition.}
	\BlankLine
	\textbf{Step 1:} Compute direct basis $b(x)$ and associated matrix $Q$ for polynomial $p(x)$.
	
	\textbf{Step 2:} Compute $\lambda^{\pm}$ and $v^{\pm}$ using \Cref{eq:eigenvaluesofdirectbasis,eq:vplus,eq:vminus}.
	
	\textbf{Step 3:} A D-SOS decomposition of $p(x)$ is given by \eqref{eq:D-SOSfordirectbasis}.
\end{algorithm}

\begin{proposition}\label{prop:degreeofDBS-DSOS}
	Let $p$ be a polynomial of $\R[x]$ with direct basis $b$, the maximal degree of D-SOS components generated by \Cref{alg:DBS-DSOS} is upper bounded by $2\deg(p)$.
\end{proposition}
\begin{proof}
	This is an immediate consequence of \Cref{prop:degreeofGS-DSOS} with $\deg(b)= \deg(p)$.
\end{proof}

Note that the number of squares produced in \Cref{alg:DBS-DSOS} is at most $2$, which is an attractive benefit of \Cref{alg:DBS-DSOS}. However, the drawback is that the maximal degree of D-SOS components is $2\deg(p)$,  twice than the minimal one $2\lceil\frac{\deg(p)}{2} \rceil$. Based on \Cref{prop:degreeofGS-DSOS}, we can consider minimizing the degree of valid basis to get a \emph{Minimal Basis Spectral D-SOS Decomposition} discussed in next subsection.

\subsection{Minimal basis spectral D-SOS decomposition}\label{subsec:MBS-DSOS}
\begin{definition}[Minimal basis]\label{def:minimal_basis}
	Let $p$ be a polynomial of $\R[x]$. Then a \emph{minimal basis} of $p$ is a valid basis whose degree is not greater than $\lceil \frac{\deg(p)}{2}\rceil$. \end{definition}

\begin{proposition}
	\label{thm:existanceofminimalbasis}
	For any polynomial $p\in \R_d[x]$, there exist minimal bases. 
\end{proposition}
\begin{proof}
	The full-basis $b(x)=(1,x_1,\ldots,x_n,x_1x_2,\ldots,x_n^{\lceil\frac{d}{2}\rceil})\tran$ with all monomials in $\R_{\lceil\frac{d}{2}\rceil}[x]$ is of course a minimal basis for any polynomial $p\in \R_{d}[x]$. Moreover, it is possible to have a subset of the full-basis as a valid basis.
\end{proof}

Note that the minimal basis is in general not unique for a polynomial, and we prefer in practice a shorter length minimal basis. E.g., $\{x_1\}$ is a minimal basis of the shortest length for $x_1^2$ and $\{x_1x_2,x_2^2 \}$ for $x_1x_2^3$. Inspired by these examples, we generalize the idea to construct a minimal basis using the \textbf{Procedure B}.

\begin{proc}{h!}{Procedure B}
	\caption{Minimal basis for polynomial $p(x)$}
	\label{proc:B}
	
	\textbf{Step 1:} Get the list of monomials for $p(x)$ as the set $M=\{m_1,\ldots,m_r\}$.
	
	\textbf{Step 2:} For each monomial $m\in M$, let $\mathcal{O}(m)$ be the index set of odd degree variables in $m$, a minimal basis of $m$ is based on the next two cases:\\
	$\rhd$ If $\mathcal{O}(m)=\emptyset$, then a minimal length minimal basis of $m(x)=c_{\alpha} x^{\alpha}$ is $K(m)=\{x^{\frac{\alpha}{2}} \}$.\\
	$\rhd$ Otherwise, we partition $\mathcal{O}(m)$ into two parts $\mathcal{O}_1(m)$ and $\mathcal{O}_2(m)$ such that:
	\begin{enumerate}
		\item[$\bullet$] $\mathcal{O}_1(m)\cap \mathcal{O}_2(m) = \emptyset$, $\mathcal{O}_1(m) \cup \mathcal{O}_2(m) = \mathcal{O}(m)$.
		\item[$\bullet$] $|\mathcal{O}_1(m)|$ and $|\mathcal{O}_2(m)|$ belong to $\{ \lceil\frac{|\mathcal{O}(m)|}{2}\rceil \}\cup \{\lfloor \frac{|\mathcal{O}(m)|}{2}\rfloor \}$. 
	\end{enumerate}
	Then a minimal length minimal basis of $m(x)=c_{\alpha}x^{\alpha}$ is 
	$$K(m) = \left\{\displaystyle x^{\lfloor\frac{\alpha}{2} \rfloor}\prod_{k\in \mathcal{O}_1(m)} x_k,\ \displaystyle x^{\lfloor\frac{\alpha}{2} \rfloor}\prod_{k\in \mathcal{O}_2(m)} x_k\right\}.$$
	\textbf{Step 3:} A minimal basis for $p(x)$ is given by 
	$b(x) = \bigcup_{m\in M} K(m).$
\end{proc}

For example, using \textbf{Procedure B} to get a minimal basis for the polynomial $x_1^2x_2^6  - 2x_1^3x_2^{100} + 10$: the set of monomials (without coefficients) is $M=\{x_1^2x_2^6, x_1^3x_2^{100},1\}$. For each monomial in $M$, we get minimal basis $\{x_1x_2^3\}$ for $x_1^2x_2^6$, $\{x_1x_2^{50},x_1^2x_2^{50}\}$ for $x_1^3x_2^{100}$, and $\{1\}$ for $1$. So that a minimal basis of $p(x)$ is 
$b(x) = \{x_1x_2^3\}\cup \{x_1x_2^{50},x_1^2x_2^{50}\} \cup \{1\} = \{1,x_1x_2^3,x_1x_2^{50},x_1^2x_2^{50}\}$,
which is much shorter than the full-basis of polynomials of two variables and of degree $\lceil \frac{103}{2}\rceil$.

Our proposed \emph{Minimal Basis Spectral D-SOS Decomposition Algorithm} (cf. MBS-DSOS) is described in \Cref{alg:MBS-DSOS}.
\begin{algorithm}[h!]
	\caption{Minimal Basis Spectral D-SOS Decomposition (MBS-DSOS)}
	\label{alg:MBS-DSOS}
	\KwIn{Polynomial $p(x)$.} 
	\KwOut{D-SOS decomposition.}
	\BlankLine
	\textbf{Step 1:} Compute $b(x)$ and $Q$ for $p(x)$ using \textbf{Procedure B}.
	
	\textbf{Step 2:} Spectral decomposition of $Q$ to get $P$ and $\Lambda$.
	
	\textbf{Step 3:} Let $r_b$ be the length of basis $b(x)$, compute $y(x)= P\tran  b(x)$ and $\mathcal{K}=\{k\in \IntEnt{1,r_b}:0\neq \lambda_k\in\Sp(Q)\}$. Then a D-SOS decomposition of $p(x)$ is given by \eqref{eq:specD-SOS}.
\end{algorithm}

\begin{proposition}\label{prop:degreeofMBS-DSOS}
	Let $p\in\R[x]$ with a minimal basis $b$. Then the maximal degree of D-SOS components generated by \Cref{alg:MBS-DSOS} equals $2\lceil \frac{\deg(p)}{2}\rceil$.
\end{proposition}
\begin{proof}
	Based on the definition of minimal basis, $\deg(b)\leq \lceil \frac{\deg(p)}{2}\rceil$. Then it follows from \Cref{prop:degreeofGS-DSOS} that the maximal degree of D-SOS components equals $2\deg(b) \leq 2\lceil \frac{\deg(p)}{2}\rceil$. On the other hand, according to \Cref{thm:universalDSOS&DCSOS} (iv), the maximal degree of D-SOS components is lower bounded by $2\lceil \frac{\deg(p)}{2}\rceil$. We conclude that the maximal degree of D-SOS components equals $2\lceil \frac{\deg(p)}{2}\rceil$.
\end{proof}

Note that applying \Cref{alg:MBS-DSOS} for quadratic form yields the spectral decomposition for quadratic form, which is both D-SOS and DC-SOS. Hence, the minimal basis spectral D-SOS decomposition can be considered as the generalization of spectral decomposition for polynomials. 

\subsection{Undominated D-SOS decomposition}\label{subsec:undominatedDSOS}
A qualification for D-SOS decomposition is the `domination' defined as follows:
\begin{definition}[Undominated D-SOS decomposition]
	A D-SOS decomposition for polynomial $p\in \R[x]$ with components $s_1$ and $s_2$ is called \emph{undominated} if and only if for any nonzero SOS polynomial $s$, the polynomials $s_1-s$ and $s_2-s$ are not SOS.   
\end{definition}

For example, the minimal basis spectral D-SOS decomposition for the quadratic form $x\tran   Q   x$ is undominated.

A natural question arises that: \emph{Is any minimal basis spectral D-SOS decomposition undominated?} The next theorem gives an affirmative answer. 
\begin{theorem}\label{thm:undom}
	Any minimal basis spectral D-SOS decomposition is undominated. 
\end{theorem}
\begin{proof}
	For any $p\in \R[x]$, let $\bar{d} = \lceil\frac{\deg(p)}{2}\rceil$. Consider the minimal basis spectral D-SOS decomposition $s_1-s_2$ associated with the full-basis $b(x)=(1,x_1,\ldots,x_n,x_1x_2,\ldots,x_n^{\bar{d}})\tran$ and the corresponding matrix $Q$. Suppose that there are $q$ nonnegative eigenvalues of $Q$ (denoted by $\lambda_1,\ldots,\lambda_q$) and $\bar{d}-q$ negative eigenvalues of $Q$ (denoted by $\lambda_{q+1},\ldots,\lambda_{\bar{d}}$), let $\Lambda^+$ and $\Lambda^-$ be the block diagonal matrices:
	$$\Lambda^+ = \left[\begin{array}{ccc|ccc}
		\lambda_1& & & & &\\
		& \ddots & & & 0_{q,\bar{d}-q}\\
		& & \lambda_q\\
		\hline
		& & & & & \\
		& 0_{\bar{d}-q,q} & & & 0_{\bar{d}-q,\bar{d}-q}& \\
		& & & & &
	\end{array} \right], 
	\Lambda^- = \left[\begin{array}{ccc|ccc}
		& & & & & \\
		& 0_{q,q}& & & 0_{q,\bar{d}-q} & \\
		& & & & & \\
		\hline
		& & & \lambda_{q+1} \\
		& 0_{\bar{d}-q,q} & & & \ddots \\
		& & & & & \lambda_{\bar{d}}
	\end{array} \right].$$
	Then $$s_1 = b(x)\tran    P   \Lambda^+   P\tran   b(x),\quad s_2 = b(x)\tran    P   (-\Lambda^-)   P\tran   b(x).$$
	It follows that, for any nonzero SOS polynomial $s$, if both $s_1-s$ and $s_2-s$ remain SOS, then $\deg(s)\leq \min \{\deg(s_1),\deg(s_2)\} = 2\lceil\frac{\deg(p)}{2}\rceil$. Furthermore, there exists a symmetric matrix $\widehat{Q}$, which is of the same size as the matrix $Q$, such that 
	$$s\in \SOS, \deg(s)\leq 2\lceil\frac{\deg(p)}{2}\rceil \implies \begin{cases}
		s(x) = b(x)\tran   \widehat{Q}  b(x)\\
		\widehat{Q}\succeq 0\\
	\end{cases},$$
	$$\left[s_1-s\in \SOS \implies P   \Lambda^+   P\tran-\widehat{Q} \succeq 0\right]\text{ and }
	\left[s_2-s\in \SOS \implies P   (-\Lambda^-)   P\tran-\widehat{Q} \succeq 0\right].$$
	Hence \begin{equation}\label{eq:undom}
		\begin{cases}
			s(x) = b(x)\tran   \widehat{Q}  b(x)\\
			\widehat{Q}\succeq 0\\
			P   \Lambda^+   P\tran-\widehat{Q} \succeq 0\\
			P   (-\Lambda^-)   P\tran-\widehat{Q} \succeq 0.
		\end{cases}
	\end{equation} 
	Let us separate $\widehat{Q}$ according to the blocks of $\Lambda^+$ and $\Lambda^-$ as:
	$$\widehat{Q} = \left[\begin{array}{c|c}
		\widehat{Q}_{1,1} & \widehat{Q}_{1,2}\\
		\hline
		\widehat{Q}_{1,2}\tran & \widehat{Q}_{2,2}   
	\end{array} \right].$$
	Then, \eqref{eq:undom} is satisfied if and only if $\widehat{Q} = 0$, which implies that:\\
	$\rhd$ $\widehat{Q}\succeq 0 \implies \widehat{Q}_{1,1}\succeq 0$ and $P   (-\Lambda^-)   P\tran-\widehat{Q} \succeq 0 \implies \widehat{Q}_{1,1}\preceq 0$, then $\widehat{Q}_{1,1}=0$.\\
	$\rhd$ By analogue, $\widehat{Q}_{2,2}=0$ is obtained from $\widehat{Q}\succeq 0$ and $P   \Lambda^+   P\tran-\widehat{Q} \succeq 0$.\\
	$\rhd$ $\widehat{Q}=0$ since $\tr(\widehat{Q}) = \tr(\widehat{Q}_{1,1}) + \tr(\widehat{Q}_{2,2}) = 0$ and $\widehat{Q}\succeq 0$.\\	
	Hence, the SOS polynomial $s$ can only be zero, implying that the minimal basis spectral D-SOS decomposition associated with the full-basis is undominated. Finally, as any minimal basis is a subset of the full-basis, $s_1$, $s_2$ and $s$ associated with any minimal basis can be presented in the full-basis, ensuring that the undomination remains valid.
\end{proof}

\subsection{Comparisons of D-SOS decomposition algorithms}
We discuss in this subsection three aspects of different D-SOS decompositions: the maximal degree of D-SOS components, the number of squares in the D-SOS decomposition, and the complexity for constructing a D-SOS decomposition. 

Firstly, the maximal degrees of D-SOS components are summarized in \Cref{tab:compareD-SOSalgos}, 
\begin{table}[h!]
	\begin{center}		
		\caption{Maximal degree of D-SOS components for D-SOS decompositions}\label{tab:compareD-SOSalgos}
		\begin{tabular}{|l|c|}  
			\hline  
			Spectral D-SOS Algorithms& Maximal degree of D-SOS components \\ 
			\hline
			\Cref{alg:GS-DSOS} (GS-DSOS) & $2\deg(b)$ \\
			\Cref{alg:DBS-DSOS} (DBS-DSOS) & $2\deg(p)$ \\
			\Cref{alg:MBS-DSOS} (MBS-DSOS) & $2\lceil \frac{\deg(p)}{2}\rceil$ \\
			\hline
		\end{tabular}
	\end{center}
\end{table}
where (MBS-DSOS) yields the minimal degree D-SOS decompositions. 

Secondly, the number of squares in D-SOS decompositions are described as follows: Let $p$ be a polynomial of $n$ variables, of degree $d$, and of $J$ monomials $m_1,\ldots,m_J$.
\begin{itemize}[leftmargin=12pt]
	\item In (GS-DSOS): Based on \eqref{eq:specD-SOS}, the number of squares is equal to the number of non-zero eigenvalues of matrix $Q$, which equals $|\mathcal{K}|$. 
	\item In (DBS-DSOS): We have at most $2$ squares since there are at most $2$ simple non-zero eigenvalues of matrix $Q$ in (DBS-DSOS) algorithm. 
	\item In (MBS-DSOS): Based on \textbf{Procedure B}, the length of the minimal basis $r_b$ can not exceed neither the length of full-basis $\binom{n+\lceil\frac{d}{2}\rceil}{n}$ nor twice of the number of monomials $2J$. Hence, the number of squares is upper bounded by $\min\{2J, \binom{n+\lceil\frac{d}{2}\rceil}{n}\}$.
\end{itemize}

Lastly, the complexity for constructing a D-SOS decomposition is mainly depending on the number of squares and the complexity to generate these squares. Note that in spectral D-SOS decompositions, Step 1 needs to solve a linear system (often large-scale and sparse) for determining the matrix $Q$, and Step 2 requires a spectral decomposition for $Q$. These computations account for a major proportion in total computation costs. Hence, for cases where the computational cost is not prohibitive, the (MBS-DSOS) will be the best practical choice due to its minimal degree D-SOS components and small number of squares. Numerical performance of these decomposition algorithms tested on randomly generated polynomial dataset will be reported in \Cref{subsec:testsonrandom}.

\section{Parity DC-SOS decomposition techniques} \label{sec:DC-SOS}
D-SOS decompositions are in general not DC-SOS. In this section, we propose some \emph{parity DC-SOS decomposition} techniques for constructing exact DC-SOS decompositions without solving SDP. 

\subsection{General parity DC-SOS decomposition}
Consider a monomial $x^{\alpha}$, it can be rebuilt by the multiplications of three elementary cases: $x_ix_j$, $x_i^{2k},k\in \N$ and $p\times q$ with $p,q\in \DC-SOS$ whose DC-SOS decompositions are summarized as follows:\\
\noindent $\rhd$ \textbf{DC-SOS decompositions for $x_ix_j$:}
\begin{equation}\label{eq:DC-SOS_elementarycase1}
	x_ix_j=\frac{1}{4}(x_i+x_j)^2-\frac{1}{4}(x_i-x_j)^2 \quad \text{ or }\quad x_ix_j=\frac{1}{2}(x_i+x_j)^2-\frac{1}{2}(x_i^2 + x_j^2).
\end{equation} 
A single variable $x_i$ can be viewed as a special case of $x_ix_j$ with $x_j=1$. The degree of DC-SOS components in \eqref{eq:DC-SOS_elementarycase1} are minimized. The first DC-SOS decomposition in $\eqref{eq:DC-SOS_elementarycase1}$ dominates the second since there exists a CSOS polynomial $\frac{1}{4}(x_i+x_j)^2$ such that $\frac{1}{2}(x_i+x_j)^2 - \frac{1}{4}(x_i+x_j)^2 = \frac{1}{4}(x_i+x_j)^2$ and $\frac{1}{2}(x_i^2+x_j^2) - \frac{1}{4}(x_i+x_j)^2 = \frac{1}{4}(x_i-x_j)^2$. \\
\noindent $\rhd$ \textbf{DC-SOS decompositions for $x_i^{2k}, k\in \N$:}
\begin{equation}\label{eq:DC-SOS_elementarycase2}
	x_i^{2k} = x_i^{2k} - 0.
\end{equation}
The degree of DC-SOS components is either $2k$ or $0$, which is a minimal degree DC-SOS decomposition. The constant polynomial $1$ can be viewed to have degree $2k=0$.\\
\noindent \textbf{$\rhd$ DC-SOS decompositions for $p\times q$ with $p,q\in \DC-SOS$:} Let $p=p_1-p_2$ and $q=q_1-q_2$ be DC-SOS decompositions of $p$ and $q$. Then applying \eqref{eq:prodDC-SOS}, we obtain a DC-SOS decomposition for $p\times q$ as:
\begin{equation}
	\label{eq:DC-SOS_elementarycase3}
	p\times q = \frac{1}{2}[(p_1+q_1)^2 + (p_2+q_2)^2] - \frac{1}{2}[(p_1+q_2)^2 + (p_2+q_1)^2],
\end{equation}
where the degree of the DC-SOS components is not greater than $$2\max\{\deg(p_1),\deg(p_2),\deg(q_1),\deg(q_2)\}.$$ Note that \eqref{eq:DC-SOS_elementarycase3} may not be a minimal degree DC-SOS decomposition.

Based on the three elementary cases, the basic idea of the \emph{General Parity DC-SOS Decomposition} (cf. GP-DSOS) is to factorize each monomial $x^{\alpha}$ as elementary cases and producing its DC-SOS decomposition using \Cref{eq:DC-SOS_elementarycase1,eq:DC-SOS_elementarycase2,eq:DC-SOS_elementarycase3}. The detailed algorithm is described in \Cref{alg:GP-DCSOS}. 
\begin{algorithm}[h]
	\caption{General Parity DC-SOS Decomposition (GP-DCSOS)}
	\label{alg:GP-DCSOS}
	\KwIn{Monomial $m(x)$.} 
	\KwOut{DC-SOS decomposition.}
	\BlankLine
	\textbf{Step 1:} Extract $x^{\alpha}$ and $c_{\alpha}$ from $m(x)$.
	
	\textbf{Step 2:} Factorize $x^{\alpha}$ as elementary cases.
	
	\textbf{Step 3:} 
	Compute a DC-SOS decomposition for $x^{\alpha}$ (using \cref{eq:DC-SOS_elementarycase1,eq:DC-SOS_elementarycase2,eq:DC-SOS_elementarycase3}) as
	$$x^{\alpha}=s_1(x) - s_2(x).$$
	
	\textbf{Step 4:} A DC-SOS decomposition of $m(x)$ is given by
	$c_{\alpha}s_1(x) - c_{\alpha}s_2(x).$
	
\end{algorithm}
Note that the degrees of the DC-SOS components generated by \Cref{alg:GP-DCSOS} depend on two important steps: (i) the factorization of $x^{\alpha}$ into elementary cases in Step 2; (ii) the order of multiplications for elementary cases in Step 3. Different factorizations may lead to different degrees, and the multiplication order of elementary cases will influence the degree of the resulting DC-SOS components as well. Next, we will discuss these issues respectively.

\paragraph{\textbf{Factorization of a monomial into elementary cases}}
We start by an example to show how the factorization of $x^{\alpha}$ influences the degree of DC-SOS components. Consider the monomial $x_1x_2$ and two factorizations: (1) we factorize as $(x_1)(x_2)$ and using \Cref{eq:DC-SOS_elementarycase1,eq:DC-SOS_elementarycase3} to get a DC-SOS decomposition as 
{\small\begin{align*}
		x_1x_2\overset{\eqref{eq:DC-SOS_elementarycase1}}{=}& \left(\frac{1}{4}(x_1+1)^2-\frac{1}{4}(x_1-1)^2 \right) \left(\frac{1}{4}(x_2+1)^2-\frac{1}{4}(x_2-1)^2 \right) \\
		\overset{\eqref{eq:DC-SOS_elementarycase3}}{=}&\frac{1}{2}\left[{{\left( \frac{{{\left( {{x}_{2}}+1\right) }^{2}}}{4}+\frac{{{\left( {{x}_{1}}+1\right) }^{2}}}{4}\right) }^{2}}+{{\left( \frac{{{\left( {{x}_{2}}-1\right) }^{2}}}{4}+\frac{{{\left( {{x}_{1}}-1\right) }^{2}}}{4}\right) }^{2}}\right]\\
		&-\frac{1}{2}\left[{{\left( \frac{{{\left( {{x}_{2}}+1\right) }^{2}}}{4}+\frac{{{\left( {{x}_{1}}-1\right) }^{2}}}{4}\right) }^{2}}+{{\left( \frac{{{\left( {{x}_{2}}-1\right) }^{2}}}{4}+\frac{{{\left( {{x}_{1}}+1\right) }^{2}}}{4}\right) }^{2}}\right],
\end{align*}}
where the degree of each DC-SOS component is $4$. (2) we factorize as $(x_1x_2)$ and using \eqref{eq:DC-SOS_elementarycase1} to get a DC-SOS decomposition as 
$$x_1x_2 = \frac{1}{4}(x_1+x_2)^2 - \frac{1}{4}(x_1-x_2)^2,$$ whose degree of each DC-SOS component is $2$ (smaller than 4). Therefore, we have to propose an optimal factorization in order to produce a minimal degree DC-SOS decomposition. The proposed factorization is described in \textbf{Procedure F} which will output a list  $L=[s_1(x),\ldots,
s_{r_L}(x)]$ of DC-SOS decompositions of length $r_L=\lceil \frac{|\alpha|}{2} \rceil$ such that $x^{\alpha} = \prod_{i=1}^{r_L}s_i$ and all $s_i$ are DC-SOS of degree $2$.

\begin{proc}{h!}{Procedure F}
	\caption{Factorization of $x^{\alpha}$}
	\label{proc:F}
	
	\textbf{Step 1:} Let $\mathcal{O}(x^{\alpha})$ be the index set of odd degree variables in $x^{\alpha}$. Then we get a factorization of $x^{\alpha}$ as $x^{\alpha}=o(x)e^2(x)$ where 
	$$o(x)=\begin{cases}
\prod_{i\in \mathcal{O}(x^{\alpha})}x_i, & \text{if }\mathcal{O}(x^{\alpha})\neq\emptyset\\
1, & \text{otherwise}\end{cases} \quad \text{ and }\quad e^2(x)=x^{\alpha}/o(x).$$
	\textbf{Step 2:} Factorize $e^2(x)$ as $(x_i^2)(x_j^2)\cdots(x_k^2)$ and produce DC-SOS decompositions for each $(x_i^2)$ using \cref{eq:DC-SOS_elementarycase2}. \\
	\textbf{Step 3:} Factorize $o(x)$ as $\prod(x_{i} x_{j})$ such that all $i,j$ belong to $\mathcal{O}(x^{\alpha})$ or $x_i=1$ or $x_j=1$, then produce DC-SOS decompositions for each couple $(x_{i} x_{j})$ using \cref{eq:DC-SOS_elementarycase1}.	 
\end{proc}

\paragraph{\textbf{Order of multiplications for elementary cases}}
We start again by an example to show how the order of multiplications influences the degree of DC-SOS components. Consider the monomial $x_1^4x_2^2x_3^2$ with two orders of multiplications: \\(1) We multiply the elementary cases in the order  $\overbrace{(\underbrace{(x_1^4x_2^2)}_{1}\underbrace{(x_3^2)}_{2})}^{3},$ where the labels $1,2$ and $3$ under/over brackets indicate the order of multiplications, which means that we compute respectively a DC-SOS decomposition for $x_1^4x_2^2$ and $x_3^2$, then produce their DC-SOS decompositions. Then, we produce a DC-SOS decomposition of $x_1^4x_2^2$ as 
\[\frac{1}{2}\left[{{\left( {{{{x}_{3}}}^{2}}+\frac{{{\left( {{{{x}_{2}}}^{2}}+{{{{x}_{1}}}^{4}}\right) }^{2}}}{2}\right) }^{2}}+\frac{{{\left( {{{{x}_{2}}}^{4}}+{{{{x}_{1}}}^{8}}\right) }^{2}}}{4}\right]-\frac{1}{2}\left[{{\left( {{{{x}_{3}}}^{2}}+\frac{{{{{x}_{2}}}^{4}}+{{{{x}_{1}}}^{8}}}{2}\right) }^{2}}+\frac{{{\left( {{{{x}_{2}}}^{2}}+{{{{x}_{1}}}^{4}}\right) }^{4}}}{4}\right]\]
with the degree $16$ for each DC-SOS component. \\(2) We multiply by the order $\overbrace{(\underbrace{(x_2^2x_3^2)}_{1}\underbrace{(x_1^4)}_{2})}^{3}$ to get a DC-SOS decomposition as
\[\frac{1}{2}\left[{{\left( \frac{{{\left( {{{{x}_{3}}}^{2}}+{{{{x}_{2}}}^{2}}\right) }^{2}}}{2}+{{{{x}_{1}}}^{4}}\right) }^{2}}+\frac{{{\left( {{{{x}_{3}}}^{4}}+{{{{x}_{2}}}^{4}}\right) }^{2}}}{4}\right] - \frac{1}{2}\left[{{\left( \frac{{{{{x}_{3}}}^{4}}+{{{{x}_{2}}}^{4}}}{2}+{{{{x}_{1}}}^{4}}\right) }^{2}}+\frac{{{\left( {{{{x}_{3}}}^{2}}+{{{{x}_{2}}}^{2}}\right) }^{4}}}{4}\right]\]
with the degree $8$ for each DC-SOS component.  

Clearly, a principal to get a small degree DC-SOS decomposition is to multiply couples with the same degree, e.g., the couple $(x_2^2x_3^2)$ is preferred to the couple $(x_1^4x_2^2)$. Hence, \textbf{Procedure M} is proposed to multiply the elements in the list $L=[s_1(x),\ldots,
s_{r_L}(x)]$ generated by \textbf{Procedure F}.
\begin{proc}{h!}{Procedure M}
	\caption{Order of Multiplications}
	\label{proc:M}
	\textbf{Step 1:} Sort the list $L$ in an increasing order by the degree of DC-SOS components.
	
	\textbf{Step 2:} \\
	\eIf{$|L|=1$}
	{\Return $L(1)$}
	{
		Compute a DC-SOS decomposition for $L(1)\times L(2)$ by \Cref{eq:DC-SOS_elementarycase3} to replace $L(1)$ and $L(2)$ as:	$L\longleftarrow (L \setminus [L(1),L(2)])\cup (L(1)\times L(2))$.
		
		\textbf{Goto Step 1}.
	}
\end{proc}

\subsection{Improved Parity DC-SOS decomposition}
By introducing \textbf{Procedure F} and \textbf{Procedure M} into \Cref{alg:GP-DCSOS}, we get an \emph{Improved Parity DC-SOS Decomposition Algorithm} (cf. IP-DCSOS) described in \Cref{alg:IP-DCSOS}. 
\setcounter{algorithm}{1}
\begin{algorithm}[h!]
	\caption{Improved Parity DC-SOS Decomposition (IP-DCSOS)}
	\label{alg:IP-DCSOS}
	\KwIn{Monomial $m(x)$.} 
	\KwOut{DC-SOS decomposition.}
	\BlankLine
	\textbf{Step 1:} Extract $x^{\alpha}$ and $c_{\alpha}$ from $m(x)$.
	
	\textbf{Step 2:} Use \textbf{Procedure F} to factorize $x^{\alpha}$ and output a list $L=[s_1(x),\ldots,s_{r_L}(x)].$
	
	\textbf{Step 3:} Use \textbf{Procedure M} to compute a DC-SOS decomposition of $x^{\alpha}$ as
	$p(x) - q(x).$
	
	\textbf{Step 4:} A DC-SOS decomposition for $m(x)$ is given by
	$c_{\alpha}p(x) - c_{\alpha}q(x).$
\end{algorithm}
\begin{proposition}\label{prop:degreeofIP-DCSOS}
	Let $m$ be a monomial of $\R[x]$. Then the degree of the DC-SOS components generated by \Cref{alg:IP-DCSOS} is equal to \begin{equation}
		\label{eq:degIP-DCSOS}
		\left\lbrace \begin{array}{ll}
			0, &\text{ if } \deg(m)=0,\\
			2, &\text{ if } \deg(m)=1,\\
			2^{\lceil \log_2(\deg(m))\rceil}, & \text{ if } \deg(m)\geq 2.  
		\end{array}\right.
	\end{equation}
\end{proposition}
\begin{proof}
	The result is obviously true for the case where $\deg(m)\in \{0,1\}$. Now, consider the case with $\deg(m)\geq 2$, we get a list of $\lceil \frac{\deg(m)}{2} \rceil$ DC-SOS polynomials of degree $2$ by \textbf{Procedure F} in Step 2. Then the loop in \textbf{Procedure M} (Step 3) will first reduce to $\lceil \frac{\deg(m)}{2^2} \rceil$ DC-SOS polynomials of degree $2^2$, and reduce again to $\lceil \frac{\deg(m)}{2^3} \rceil$ DC-SOS polynomials of degree $2^3$ etc, until there is only one element in $L$ of degree $2^k$, i.e., $\lceil \frac{\deg(m)}{2^{k}} \rceil=1$. It follows that $k=\lceil \log_2(\deg(m))\rceil$ and the degree of DC-SOS components equals $2^{\lceil \log_2(\deg(m))\rceil}$.
\end{proof}

Note that the degree of DC-SOS components in \Cref{prop:degreeofIP-DCSOS} is not minimized since $2^{\lceil \log_2(\deg(m))\rceil} > 2\lceil \frac{\deg(m)}{2}\rceil$ when $\deg(m)\notin 2^{\N}$. The degree's gap is limited of order $O(\deg(m))$. The reason of this defect lies in the use of \cref{eq:prodDC-SOS} whose DC-SOS components are not of minimal degree. In the next subsection, we propose a minimal degree DC-SOS decomposition algorithm.

\subsection{Minimal degree DC-SOS decomposition}
\begin{theorem}\label{thm:dcsosformulation}
	Let $m\in \N^*, n\in \N^*$, $x\in \mathbb{R}^n, \alpha\in \mathbb{N}^n$, $\calN=\IntEnt{1,n}$ and  
	$\binom{m}{\alpha}$ the multinomial coefficient defined by
	$$\binom{m}{\alpha} = \frac{m!}{\alpha_1!\alpha_2!\cdots\alpha_n!}$$
	under the convention that $\binom{m}{\alpha}=0$ if $\alpha\in (\mathbb{N^*})^n$ and $|\alpha|=m<n$. Then   
	\begin{equation}\label{eq:factorization}
		\sum_{\calA \subseteq \calN, \calA\neq \emptyset} (-1)^{\left| \calA \right|} \left(\sum_{j \in \calA} x_j\right)^m = (-1)^n \sum_{|\alpha|=m,\alpha\in (\mathbb{N^*})^n}\binom{m}{\alpha}x^{\alpha}.	
	\end{equation}
\end{theorem}
\begin{proof}
	Let $f(x)$ be the left part of \Cref{eq:factorization} which is clearly a homogeneous polynomial of degree $m$. Firstly, we prove that $f(x)$ is a multiple of $\prod_{i\in \calN}x_i$: If $x_1=0$ then $f(x)=0$, because in each term $(-1)^{|\calA|}\left(\sum_{j \in \calA} x_j\right)^m$, if $1\in \calA$, then $x_1=0$ will turn the term to $(-1)^{|\calA|}\left(\sum_{j \in \calA\setminus\{1\}} x_j\right)^m$ which cancels out the term $(-1)^{|\calA\setminus\{1\}|}\left(\sum_{j \in \calA\setminus\{1\}} x_j\right)^m$ in $f(x)$ with opposite sign; otherwise, we can add $x_1$ to get $(-1)^{|\calA|}\left(\sum_{j \in \calA\cup\{1\}} x_j\right)^m$ which cancels out the term $(-1)^{|\calA\cup\{1\}|}\left(\sum_{j \in \calA\cup\{1\}} x_j\right)^m$ in $f(x)$ with opposite sign. So $x_1$ is a divisor of $f(x)$ and by the same argument for $x_2, \dotsc, x_n$. Hence $f(x)$ is a multiple of $\prod_{i\in \calN}x_i$. Next, we have three cases:\\ 
	(i) If $m<n$, then $\deg(f(x))=m<\deg(\prod_{i\in \calN}x_i)=n$. An $m$ degree polynomial $f(x)$ is a multiple of an $n$ degree polynomial $\prod_{i\in \calN}x_i$ with $m<n$ implies that $f(x) = 0.$\\
	(ii) If $m=n$, then $\deg(f(x))=\deg(\prod_{i\in \calN}x_i)=n$. In this case, the polynomial $f(x)$ as a multiple of $\prod_{i\in \calN}x_i$ must be of the form $c\times \prod_{i\in \calN}x_i$ with $c\in \Z$. This term can be only provided by the term $(-1)^{|\calN|}(\sum_{j\in \calN}x_j)^m$ whence $c=(-1)^n n!$ and 
		\begin{equation}
			f(x) = (-1)^n n! \prod_{i\in \calN}x_i.
		\end{equation}
	(iii) If $m>n$, then $f(x)$ must be of the form $\sum_{|\alpha|=m,\alpha\in (\mathbb{N}^*)^n}c_{\alpha} x^{\alpha}$ since $f(x)$ is a homogeneous polynomial of degree $m$ and a multiple of $\prod_{i\in \calN}x_i$. Thus each $x^{\alpha}$ is a multiple of $\prod_{i\in \calN}x_i$ which can be only provided by the term $(-1)^{|\calN|}(\sum_{j\in \calN}x_j)^m$. Expanding this by multinomial theorem, we get
		$$(-1)^{|\calN|}(\sum_{j\in \calN}x_j)^m = (-1)^{n}\sum_{|\alpha|=m,\alpha\in (\mathbb{N}^*)^n} \binom{m}{\alpha} x^{\alpha},$$
		i.e., $c_{\alpha} = (-1)^n \binom{m}{\alpha}$. Thus 
		$$f(x) = (-1)^n \sum_{|\alpha|=m,\alpha\in (\mathbb{N^*})^n}\binom{m}{\alpha}x^{\alpha}.$$
	We conclude from (i), (ii) and (iii) the same identity \cref{eq:factorization} for all $m$ and $n$ in $\N^*$.
\end{proof}

\begin{corollary}\label{cor:dcsosforprodpi}
	Let $n\in \N^*$, $\calN = \IntEnt{1,n}$, and $p_i,i\in \calN$ be $n$ CSOS polynomials. Then we have a DC-SOS decomposition for $\prod_{i\in \calN} p_i$ as: 
	\begin{equation}\label{eq:DC-SOSdecompwithsamedegree}
		\prod_{i\in \calN}p_i = \frac{1}{n!} \sum_{\calA \subseteq \calN, \calA\neq \emptyset} (-1)^{\left| \calA\right|+n} \left(\sum_{j \in \calA} p_j\right)^n.
	\end{equation}
	Moreover, if $\deg(p_i)$ are all equal for all $i \in \mathcal{N}$, then 
	\begin{itemize}[leftmargin=12pt]
		\item for any nonempty subset $\calA$ of $\calN$, we have $$\displaystyle\deg\left(\left(\sum_{j \in \calA} p_j\right)^n\right)=\displaystyle\deg\left(\prod_{i\in\calN} p_i\right).$$
		\item \cref{eq:DC-SOSdecompwithsamedegree} provides a minimal degree DC-SOS decomposition for $\prod_{i\in \calN} p_i$.
	\end{itemize}
\end{corollary}
\begin{proof}
	We replace in identity \cref{eq:factorization} all $x_i$ by $p_i$ with $m=n$ to get 
	$$
	\prod_{i=1}^{n}p_i = \frac{1}{n!} \sum_{\calA \subseteq \calN} (-1)^{\left| \calA\right|+n} \left(\sum_{j \in \calA} p_j\right)^n.
	$$
	Clearly, each term $\left(\sum_{j \in \calA} p_j\right)^n$ is CSOS, then the right part of formulation \cref{eq:DC-SOSdecompwithsamedegree} is DC-SOS. Moreover, if $\deg(p_i), i\in \calN$ are all equal, then for all $\calA\subset \calN, \calA \neq\emptyset$, we get $$\deg\left(\left(\sum_{j \in \calA} p_j\right)^n\right) =\deg\left(\prod_{i\in \calN} p_i\right) = n\deg(p_i), \forall i\in \calN,$$
	which implies that \cref{eq:DC-SOSdecompwithsamedegree} provides a minimal degree DC-SOS decomposition for $\prod_{i\in \calN} p_i$. 
\end{proof}

Based on \Cref{cor:dcsosforprodpi}, we propose a \emph{Minimal Degree DC-SOS Decomposition Algorithm} (cf. MD-DCSOS) described in \Cref{alg:MD-DCSOS}.
\begin{algorithm}[h!]
	\caption{Minimal Degree DC-SOS Decomposition (MD-DCSOS)}
	\label{alg:MD-DCSOS}
	\KwIn{Monomial $m(x)$.} 
	\KwOut{DC-SOS decomposition.}
	\BlankLine
	\textbf{Step 1:} Extract $x^{\alpha}$ and $c_{\alpha}$ for $m(x)$.
	
	\textbf{Step 2:} Use \textbf{Procedure F} to factorize $x^{\alpha}$ and produce a list of $r$ DC-SOS decompositions $L=[g_1(x)-h_1(x),\ldots,g_{r}(x)-h_{r}(x)].$
	
	\textbf{Step 3:} For each $\calA\subseteq \calR:=\IntEnt{1,r}$, use \cref{eq:DC-SOSdecompwithsamedegree} to get a DC-SOS decomposition:    
	$$\prod_{i\in \calR\setminus \calA } g_i \prod_{j\in \calA } h_j = \underbrace{\bar{g}_{\calA}(x)}_{\CSOS} - \underbrace{\bar{h}_{\calA}(x)}_{\CSOS}.$$
	
	\textbf{Step 4:} A DC-SOS decomposition for $m(x)$ is given by
	$$c_{\alpha}\sum_{\calA\subseteq \calR} (-1)^{|\calA|} (\bar{g}_{\calA}(x) - \bar{h}_{\calA}(x)).$$
	
\end{algorithm}
\begin{proposition}\label{prop:degreeofMD-DCSOS}
	Let $m$ be a monomial of $\R[x]$. Then the degree of DC-SOS components generated by \Cref{alg:MD-DCSOS} equals $2\lceil\frac{\deg(m)}{2}\rceil$.
\end{proposition}

\begin{proof}
	\textbf{Procedure F} in Step 2 factorizes $m$ and produces a list of $r$ ($=\lceil\frac{\deg(m)}{2}\rceil$) DC-SOS polynomials $[g_1-h_1,\ldots,g_r-h_r]$ with all $g_i$ and $h_i$ being quadratic convex polynomials or zeros. Let $\calR=\IntEnt{1,r}$. Then 
	$$m(x)=\prod_{i=1}^{r}g_i(x)-h_i(x) = \sum_{\calA \subseteq \calR} (-1)^{|\calA|} \prod_{i\in \calR\setminus \calA } g_i \prod_{j\in \calA } h_j.$$
	Using formulation \cref{eq:DC-SOSdecompwithsamedegree}, we get a DC-SOS decomposition for each $\prod_{i\in \calR\setminus \calA } g_i \prod_{j\in A } h_j$ as $\bar{g}_{\calA}-\bar{h}_{\calA}$ with DC-SOS components of degree $2r$ if $h_j\neq 0, \forall j\in \calA$, or $0$ if $\exists j\in \calA, h_j=0$. Hence, the degree of DC-SOS components for $m$ is equal to $2r = 2\lceil\frac{\deg(m)}{2}\rceil$.
\end{proof}

\subsection{Comparisons of DC-SOS decomposition algorithms}
Since all methods are dealing with monomials only, and a polynomial can be computed by linear combination of DC-SOS decompositions of monomials. For this reason, all proposed DC-SOS decompositions can be executed in parallel.

Firstly, the maximal degrees of DC-SOS components are summarized in \Cref{tab:compareDC-SOSalgos} where (MD-DCSOS) provides minimal degree DC-SOS decompositions, whereas the degree of DC-SOS decomposition generated by (GP-DCSOS) depends on the factorization of $x^{\alpha}$ and the order of multiplications. 
\begin{table}[H]
	\begin{center}		
		\caption{Maximal degree of DC-SOS components for DC-SOS decompositions}\label{tab:compareDC-SOSalgos}
		\begin{tabular}{|l|c|}  
			\hline  
			DC-SOS Algorithms & Maximal degree of DC-SOS components \\ 
			\hline
			\Cref{alg:GP-DCSOS} (GP-DCSOS) & -- \\ 
			\Cref{alg:IP-DCSOS} (IP-DCSOS) & \eqref{eq:degIP-DCSOS}\\
			\Cref{alg:MD-DCSOS} (MD-DCSOS) & $2\lceil \frac{\deg(m)}{2}\rceil$ \\
			\hline
		\end{tabular}
	\end{center}
\end{table}

Secondly, let $p$ be a polynomial with $J$ monomials $m_1,\ldots,m_J$. Then the number of squares in DC-SOS decompositions are:
\begin{itemize}[leftmargin=12pt]
	\item In (GP-DCSOS) and (IP-DCSOS): For each monomial, \cref{eq:DC-SOS_elementarycase3} is used to produce DC-SOS decompositions of the multiplications of DC-SOS polynomials, which yields $4$ squares (namely, outer squares) in the final decomposition. Then for polynomial $p$ with $J$ monomials, there are $4J$ outer squares. Note that each outer square should be also a CSOS polynomial (namely, inner squares), and the number of inner squares depends on the number of couples in the factorization of $x^{\alpha}$. 
	\item In (MD-DCSOS): For each monomial $m_i, i\in \IntEnt{1,J}$, \textbf{Procedure F} factorizes $m_i$ to a list of $r_i=\lceil\frac{\deg(m_i)}{2} \rceil$ DC-SOS polynomials in Step 2, Then, for $\calR=\IntEnt{1,r_i}$, there are $2^{r_i} = 2^{\lceil \frac{\deg(m_i)}{2}\rceil}$ subsets $\calA$, and based on \cref{eq:DC-SOSdecompwithsamedegree}, there are at most $2^{r_i}$ squares for each $\calA$ which yields $4^{r_i}$ squares for $m_i$. Then, for polynomial $p$, we have at most $\sum_{i=1}^{J}4^{\lceil\frac{\deg(m_i)}{2}\rceil}$ squares.
\end{itemize}

Lastly, the complexity for DC-SOS decompositions depends as well on the number of squares and the complexity for computing these squares.  (IP-DCSOS) could be faster than (MD-DCSOS) due to the less number of squares, but the degree of the DC-SOS components are not minimized. In practice, if the degree of the DC-SOS components is not important or $\deg(m_i)\in 2^{\N}$, then (IP-DCSOS) should be the best choice in terms of less number of squares; otherwise, if a minimal degree decomposition is required, then (MD-DCSOS) is recommended. 

\section{Numerical simulations}\label{sec:simulations}
In this section, we report some numerical results on the real performance of the D-SOS and DC-SOS decomposition algorithms. 
Our codes are implemented in MATLAB (R2020a), namely $\POLYDSOS$ and \POLYDCSOS, available at \url{https://github.com/niuyishuai/POLY2DSOS_DCSOS}, and tested on a cluster equipped with CentOS Linux 7 64-bit, $60$ CPUs (Intel Xeon Gold 6148 CPU @ 2.40GHz) and 256GB of RAM.
\paragraph{\textbf{Polynomial modeling}} For efficiently modeling multivariate polynomials on MATLAB, we use POLYLAB  \cite{Polylab} (a multivariate polynomial modeling toolbox, code available at \url{https://github.com/niuyishuai/Polylab}). This tool is chosen for its high efficiency on dense and large-scale polynomial generations and operations. It is $100$ times faster than the MATLAB symbolic toolbox, and $5$ times faster than YALMIP \cite{Yalmip} on polynomial constructions.

\paragraph{\textbf{Test dataset}}
For each $n\in \{2,5,8,11,14,17,20\}$, $d\in \{2,3,4,5,6\}$ and polynomial density $\density\in \{0.2,0.4,0.6,0.8,1\}$, we generate $10$ polynomials whose coefficients are integers chosen in the interval $[-10,10]$ (i.e., $1750$ polynomials in total). These settings cover most real-world polynomial optimization problems. 

\subsection{Performance of D-SOS and DC-SOS decompositions}\label{subsec:testsonrandom}
In this subsection, we compare numerical performance among four algorithms: two D-SOS decomposition algorithms (DBS-DSOS) and (MBS-DSOS), and two DC-SOS decomposition algorithms (IP-DCSOS) and (MD-DCSOS).
The general versions (GS-DSOS) and (GP-DCSOS) are not tested due to their flexibilities. We are interested in the impact to the performance from three aspects: the number of variables $n$, the polynomial degrees $d$ and the polynomial density $\density$, versus the natural logarithm of total decomposition time. The maximal degrees of components are also compared.

\paragraph{\textbf{Comments on the numerical results}}
\Cref{fig:perfDSOSDCSOS} demonstrates the performance of the compared D-SOS and DC-SOS decomposition algorithms.
\begin{figure}[htb!]
	\centering
	\subfigure[number of variables v.s. log time]{
		\label{subfig:dsos_dcsos_nvstime} 
		\includegraphics[width=0.46\textwidth]{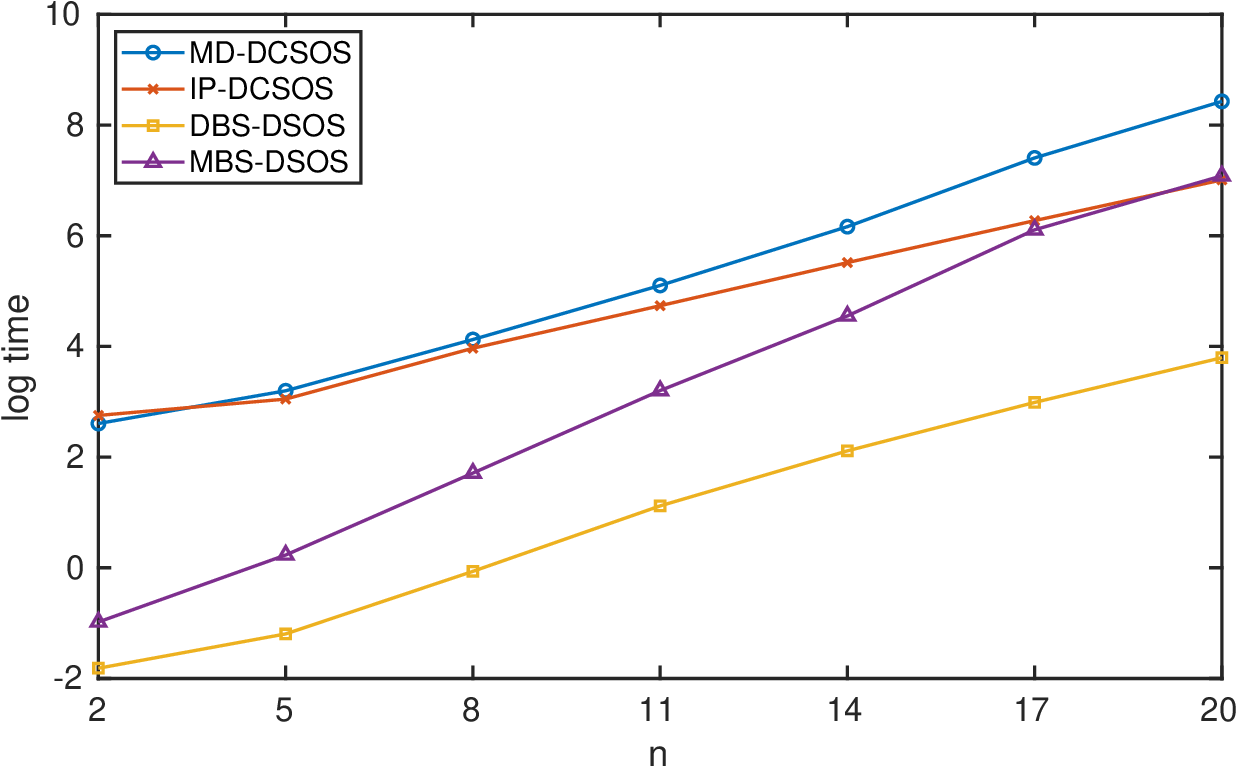}
	}
	\subfigure[polynomial degree v.s. log time]{ 
		\label{subfig:dsos_dcsos_dvstime} 
		\includegraphics[width=0.46\textwidth]{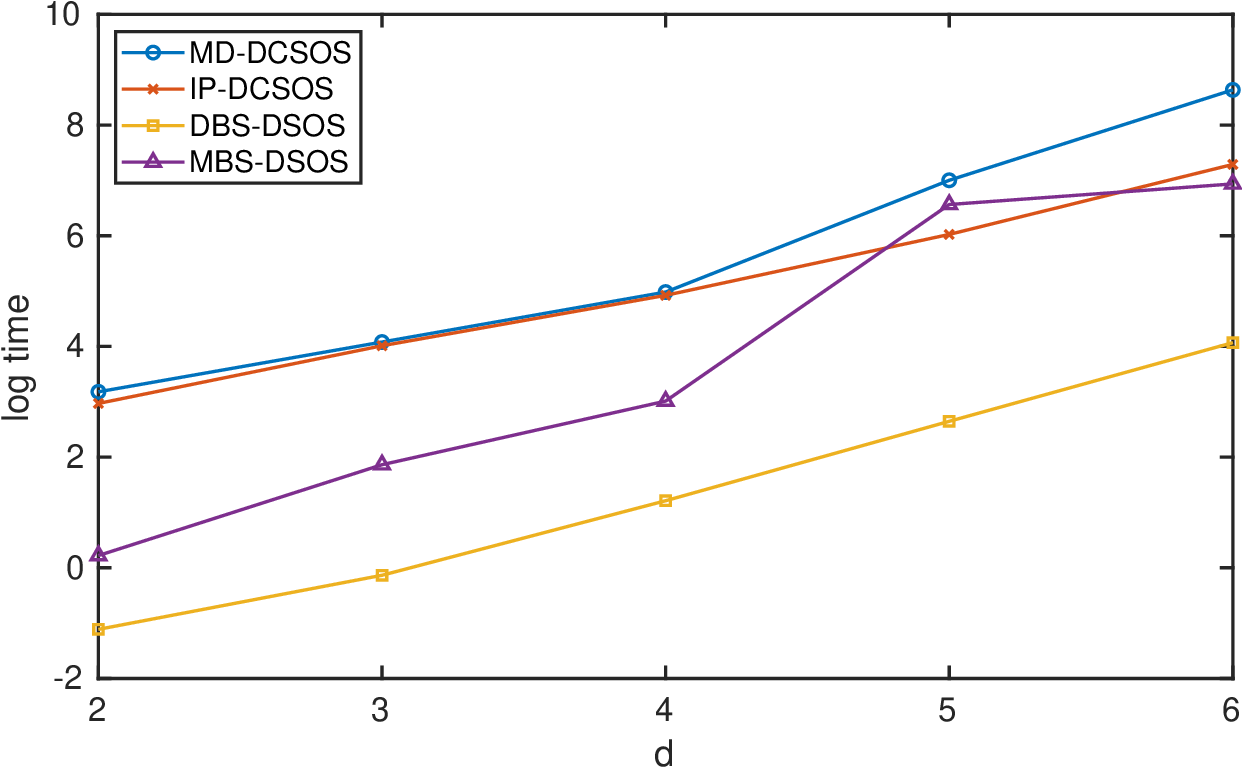}
	}
	\subfigure[polynomial density v.s. log time]{ 
		\label{subfig:dsos_dcsos_denvstime} 
		\includegraphics[width=0.46\textwidth]{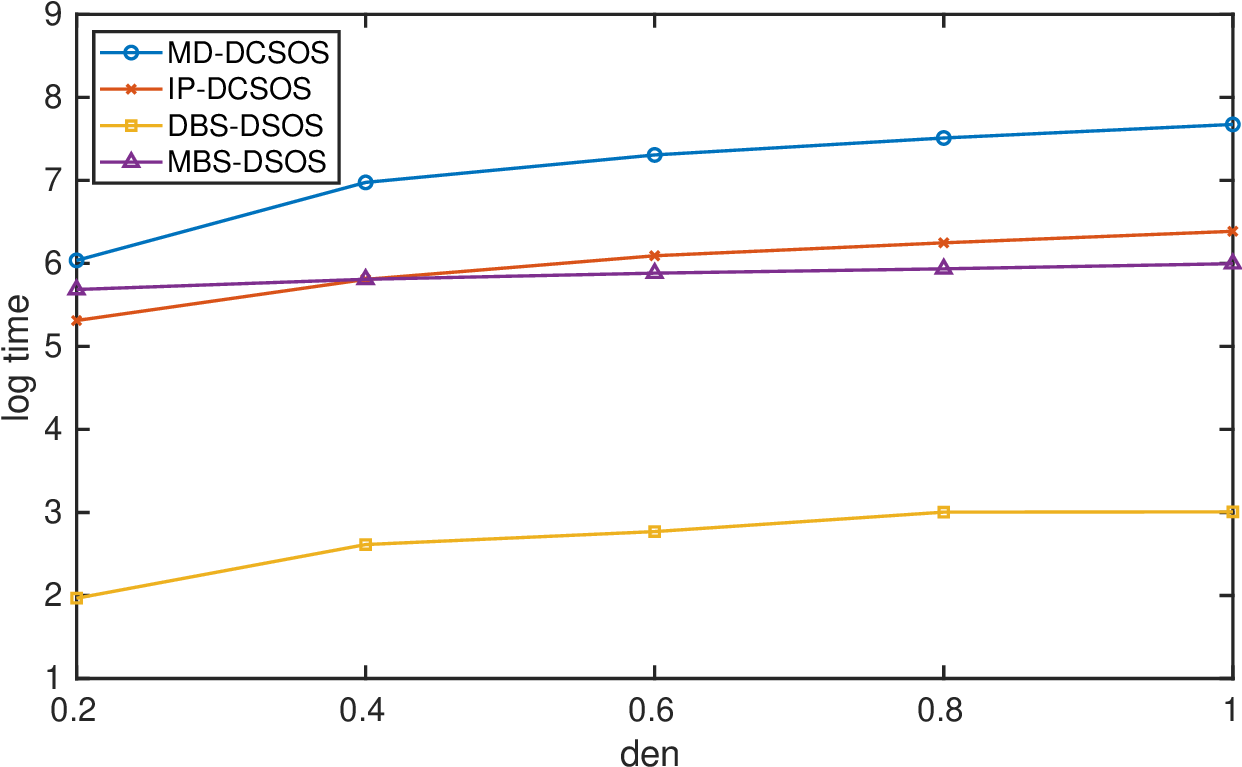} 
	}
	\subfigure[maximal degree of components]{ 
		\label{subfig:dsos_dcsos_maxdeg} 
		\includegraphics[width=0.46\textwidth]{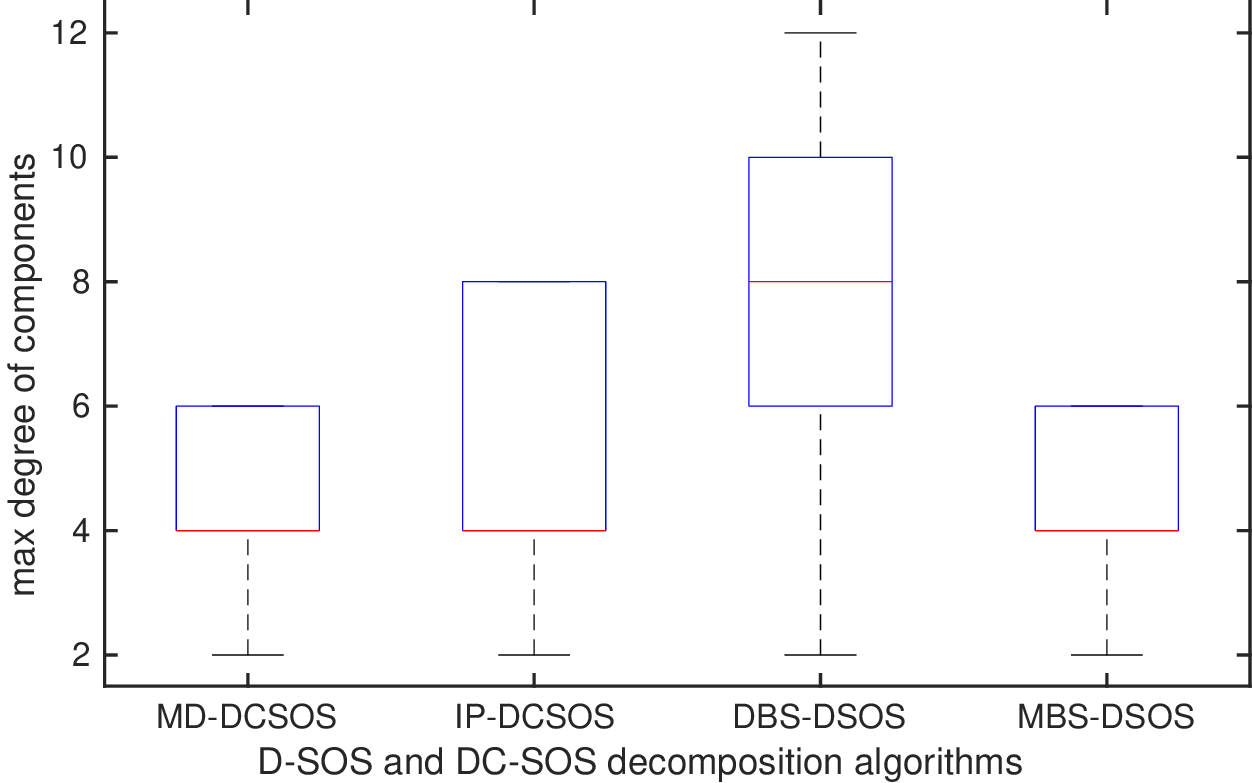} 
	}
	\caption{Numerical performance of D-SOS and DC-SOS decomposition algorithms}
	\label{fig:perfDSOSDCSOS}
\end{figure}
\begin{itemize}[leftmargin=12pt]
	\item Concerning the decomposition time: the fastest method is (DBS-DSOS) and the slowest one is (MD-DCSOS). In general, the D-SOS decompositions are faster than the DC-SOS decompositions. For all methods, the log decomposition time grows almost linearly over $n$ (\Cref{subfig:dsos_dcsos_nvstime}) and $d$ (\Cref{subfig:dsos_dcsos_dvstime}), grows slightly concavely over the polynomial density (\Cref{subfig:dsos_dcsos_denvstime}).
	\item Concerning the maximal degree of components: (MBS-DSOS) and (MD-DCSOS) provided minimal degree D-SOS and DC-SOS components as proved in \Cref{prop:degreeofMD-DCSOS,prop:degreeofMBS-DSOS}; The maximal degree generated by (IP-DCSOS) is of form $2^{\N}$ as proved in \Cref{prop:degreeofMD-DCSOS}, so that in some tested cases, the maximal degree is higher than (MD-DCSOS). However, the average maximal degree for (IP-DCSOS) on the test dataset is the same to (IP-DCSOS) and (MBS-DSOS), whereas, (DBS-DSOS) leads to an average degree twice greater than the minimal degree.	
\end{itemize}

\subsection{Performance of the parallel DC-SOS decompositions} DC-SOS decompositions can be parallelized based on the decomposition of each monomial independently. In this subsection, we report numerical performance of parallel DC-SOS decomposition algorithms for (IP-DCSOS) and (MD-DCSOS) implemented using MATLAB \texttt{parfor}. The performance of the parallelism is measured by the Speedup Ratio \cite{Eager1989} defined by $S(N) = T_1/T_N$
with $T_1$ and $T_N$ being the wall-clock time for sequential code with one processor and for parallel code with $N$ processors. 

\begin{figure}[h!]
	\centering
	\subfigure[Speedup ratio of (IP-DCSOS)]{
		\label{subfig:dcsos_ip_speedup} 
		\includegraphics[width=0.46\textwidth]{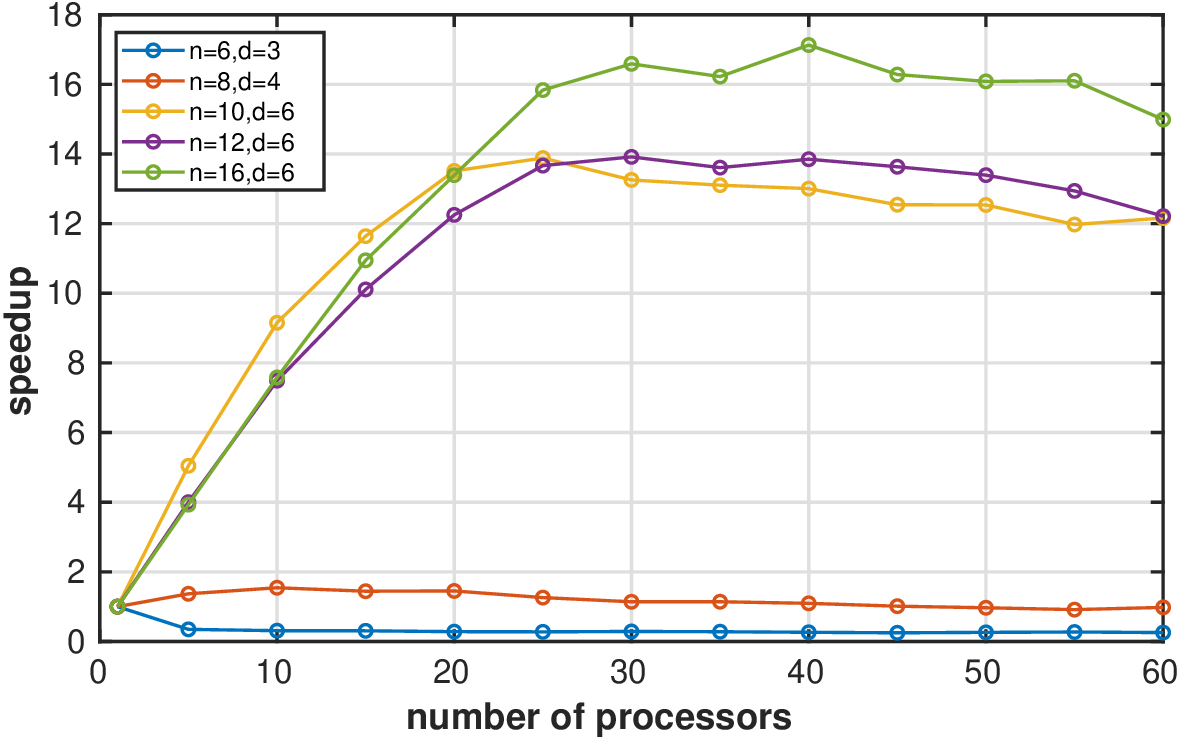}
	}
	\subfigure[Speedup ratio of (MD-DCSOS)]{
		\label{subfig:dcsos_md_speedup} 
		\includegraphics[width=0.46\textwidth]{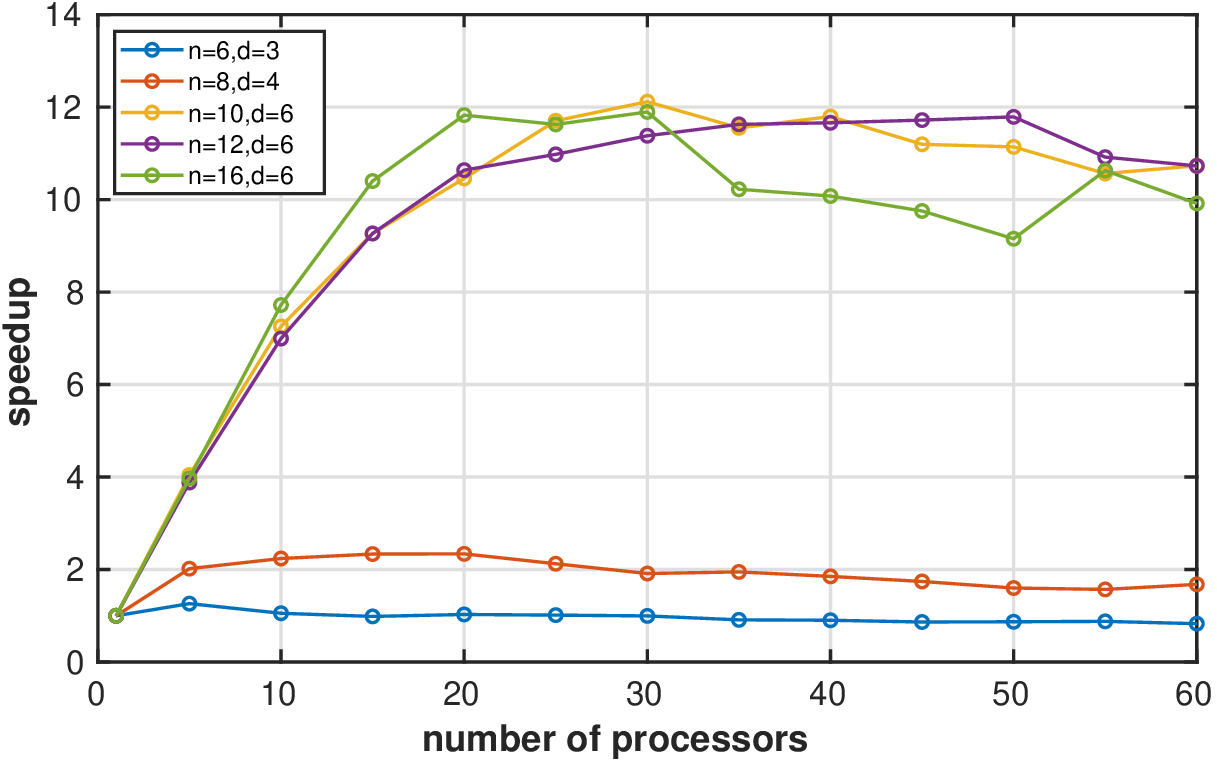}
	}
	\caption{Speedup ratios of the parallel (IP-DCSOS) and the parallel (MD-DCSOS) for full-basis polynomials using up to $60$ processors}
	\label{fig:parallel_perf_dcsos}
\end{figure}

\paragraph{\textbf{Comments on the performance of the parallelism}}
\Cref{fig:parallel_perf_dcsos} summarized the real performance of two parallel DC-SOS decomposition algorithms for dense ($\density=1$) polynomials using up to $60$ processors. Specifically, for each setting $(n,d)$, we test the polynomial of full-basis, i.e., the polynomial $\sum_{|\alpha|=d} x^{
\alpha}$ with $x\in \R^n$. We can observe that:
\begin{itemize}[leftmargin=12pt]
	\item More speedup appears for large-scale polynomials than for small-scale cases as shown in \Cref{subfig:dcsos_ip_speedup,subfig:dcsos_md_speedup}. Hence, the speedup seems to be more prominent for large-scale polynomials. 
	\item The speedup typically increases up to a maximum and then decreases with the increase of the number of processors. E.g., for polynomials with $(n,d)=(16,6)$, we obtained about $18$ fold maximal speedup for (IP-DCSOS) using about $40$ processors, and $12$ fold maximal speedup for (MD-DCSOS) using about $20$ processors. Particularly, for large-scale polynomials with large number of processors, e.g., (MD-DCSOS) with $(n,d)=(16,6)$ using more than $30$ processors, the speedup becomes unstable, which decreases from $n=30$ to $n=50$ and increases from $n=50$ to $n=55$. This observation is probably caused by the intensive communication and synchronization costs among the machines in the cluster. 
\end{itemize}

It worth noting that the DC-SOS decompositions will be very useful to reformulate any polynomial optimization problem as a DC problem. There are several successful real-world applications using the DC-SOS decomposition techniques presented in our paper such as the quadratic eigenvalue complementarity problems with DCA \cite{niu2019improved} and the Mean-Variance-Skewness-Kurtosis high-order moment portfolio optimization problem with Boosted-DCA  \cite{Paper_Niu2019higher}. Readers interested in more details about how to apply DC-SOS in real-world applications and its numerical performance are referred to these articles. 

\section{Conclusions and Perspectives}\label{sec:conclusion}
In this paper, we introduced the set of D-SOS and DC-SOS polynomials, and proved that any polynomial can be reformulated as D-SOS and DC-SOS with any desired precision in polynomial time based on SDP. Some algebraic properties on these sets of polynomials are investigated. We proposed several exact D-SOS and DC-SOS decomposition algorithms, as well as their parallel versions dealing with monomials. Numerical tests demonstrated the real performance of D-SOS and DC-SOS algorithms tested on randomly generated polynomials. We hope that our decomposition techniques will provide a useful tool for solving real-world large-scale polynomial optimization problems via DC programming approaches. 

Concerning future work, several problems deserve more attention: (1) DC-SOS decomposition with minimal number of squares. Note that if a DC-SOS decomposition consists of a large number of squares, solving the corresponding convex subproblems in DCA-based methods will be inefficient. (2) Undominated DC-SOS decomposition. It is still unclear how to generate an undominated DC-SOS decomposition (except for the quadratic form) for general polynomials. The undominated DC-SOS decomposition corresponds to the best convex overestimation for polynomial functions, and will lead to fewer iterations in DCA. (3) Power-sum DC-SOS decomposition, i.e., DC-SOS decomposition with a valid basis in the form of powers of linear forms. We observed that such basis exists for all polynomials, which may lead to DC-SOS decompositions with a sparse structure and a small number of squares. (4) Combination of DC-SOS decomposition with DCA and Lasserre's SDP relaxation. The Lasserre's hierarchy \cite{Paper_Lasserre2001,Paper_Lasserre2008,Paper_Lasserre2009} is a well-known elegant tool for polynomial optimization. However, it may require a high order SDP relaxation for a good approximation which is computationally intractable. Hence, the use of a low-order SDP relaxation to start DCA with a DC-SOS decomposition could give a better local solution in a shorter computation time. (5) Solving polynomial optimization involving D-SOS decompositions, namely \emph{D-SOS program}, seems to be novel. The spectral D-SOS decomposition was proved to be undominated in this paper and could be constructed more effectively than the DC-SOS decompositions. Although a sparse DC constrained DC formulation of the D-SOS program exists, as introduced in \Cref{sec:intro}, the question of how to solve the D-SOS program directly, by leveraging the SOS structure without relying on DC and convexity, merits more attention from the broader optimization community.

\section*{Acknowledgments}
The first author is supported by the Natural Science Foundation of China (Grant No: 11601327), the Key Construction National ``$985$" Program of China (Grant No: WF220426001), and the Research Grants Council (Grant No: 15304019). Special thanks to the editors and anonymous reviewers for their meticulous review of our manuscript. Their useful comments and suggestions helped us to significantly improve the quality of the paper. Additionally, we express our profound gratitude to Defeng Sun and Xun Li for their generous financial support, which has been instrumental in advancing the first author's research endeavors.
\bibliographystyle{siamplain}
\bibliography{references}

\begin{thebibliography}{10}

\bibitem{Paper_Ahmadi2017}
{\sc A.~A. Ahmadi and G.~Hall}, {\em Dc decomposition of nonconvex polynomials
  with algebraic techniques}, Math. Program., Ser. B, 169 (2017), pp.~1--26.

\bibitem{ahmadi2019dsos}
{\sc A.~A. Ahmadi and A.~Majumdar}, {\em Dsos and sdsos optimization: more
  tractable alternatives to sum of squares and semidefinite optimization}, SIAM
  Journal on Applied Algebra and Geometry, 3 (2019), pp.~193--230.

\bibitem{Paper_Ahmadi2013}
{\sc A.~A. Ahmadi, A.~Olshevsky, P.~A. Parrilo, and J.~N. Tsitsiklis}, {\em
  Np-hardness of deciding convexity of quartic polynomials and related
  problems}, Math. Program., 137 (2013), pp.~453--476.

\bibitem{Paper_Ahmadi2011}
{\sc A.~A. Ahmadi and P.~A. Parrilo}, {\em A complete characterization of the
  gap between convexity and sos-convexity}, SIAM J. Optim., 23 (2011),
  pp.~811--833.

\bibitem{Paper_Ahmadi2012}
{\sc A.~A. Ahmadi and P.~A. Parrilo}, {\em A convex polynomial that is not
  sos-convex}, Math. Program., 135 (2012), pp.~275--292.

\bibitem{BDCA_S}
{\sc F.~J.~A. Artacho, R.~M. Fleming, and P.~T. Vuong}, {\em Accelerating the
  dc algorithm for smooth functions}, Mathematical Programming, 169 (2018),
  pp.~95--118.

\bibitem{bakonyi1995euclidian}
{\sc M.~Bakonyi and C.~R. Johnson}, {\em The euclidian distance matrix
  completion problem}, SIAM Journal on Matrix Analysis and Applications, 16
  (1995), pp.~646--654.

\bibitem{DSDP}
{\sc S.~J. Benson, Y.~Ye, and X.~Zhang}, {\em Solving large-scale sparse
  semidefinite programs for combinatorial optimization}, SIAM Journal on
  Optimization, 10 (2000), pp.~443--461.

\bibitem{Book_Bertsekas2009}
{\sc D.~P. Bertsekas}, {\em Convex Optimization Theory}, Athena Scientific,
  1st~ed., 2009.

\bibitem{Biosca2001}
{\sc A.~F. Biosca}, {\em Representation of a polynomial function as a
  difference of convex polynomials, with an application}, in Generalized
  Convexity and Generalized Monotonicity, Springer, 2001, pp.~189--207.

\bibitem{Paper_Blekherman2009}
{\sc G.~Blekherman}, {\em Convex forms that are not sums of squares}, arXiv
  preprint arXiv:0910.0656,  (2009).

\bibitem{CSDP}
{\sc B.~Borchers}, {\em Csdp, a c library for semidefinite programming},
  Optimization Methods \& Software, 11 (1999), pp.~613--623,
  \url{https://projects.coin-or.org/Csdp/}.

\bibitem{boyd2007tutorial}
{\sc S.~Boyd, S.-J. Kim, L.~Vandenberghe, and A.~Hassibi}, {\em A tutorial on
  geometric programming}, Optimization and engineering, 8 (2007), pp.~67--127.

\bibitem{de2019inertial}
{\sc W.~de~Oliveira and M.~P. Tcheou}, {\em An inertial algorithm for dc
  programming}, Set-Valued and Variational Analysis, 27 (2019), pp.~895--919.

\bibitem{Dur2013testing}
{\sc M.~D\"{u}r and J.-B. Hiriart-Urruty}, {\em Testing copositivity with the
  help of difference-of-convex optimization}, Math. Program., 140 (2013),
  pp.~31--43.

\bibitem{Eager1989}
{\sc D.~L. {Eager}, J.~{Zahorjan}, and E.~D. {Lazowska}}, {\em Speedup versus
  efficiency in parallel systems}, IEEE Transactions on Computers, 38 (1989),
  pp.~408--423, \url{https://doi.org/10.1109/12.21127}.

\bibitem{gaudioso2018minimizing}
{\sc M.~Gaudioso, G.~Giallombardo, G.~Miglionico, and A.~M. Bagirov}, {\em
  Minimizing nonsmooth dc functions via successive dc piecewise-affine
  approximations}, Journal of Global Optimization, 71 (2018), pp.~37--55.

\bibitem{Gurobi}
{\sc Gurobi}, {\em Gurobi 9.0}, \url{http://www.gurobi.com/}.

\bibitem{Paper_Hartman1959}
{\sc P.~Hartman}, {\em On functions representable as a difference of convex
  functions}, Pacific Journal of Mathematics, 9 (1959), pp.~707--713.

\bibitem{Paper_Helton2010}
{\sc J.~W. Helton and J.~Nie}, {\em Semidefinite representation of convex
  sets}, Mathematical Programming, 122 (2010), pp.~21--64.

\bibitem{Paper_Hilbert1888}
{\sc D.~Hilbert}, {\em \"{U}ber die darstellung definiter formen als summe von
  formenquadraten}, Math. Ann., 32 (1888), pp.~342--350.

\bibitem{Cplex}
{\sc IBM}, {\em Ibm ilog cplex optimization studio 12.7},
  \url{https://www.ibm.com/us-en/marketplace/ibm-ilog-cplex}.

\bibitem{joki2017proximal}
{\sc K.~Joki, A.~M. Bagirov, N.~Karmitsa, and M.~M. M{\"a}kel{\"a}}, {\em A
  proximal bundle method for nonsmooth dc optimization utilizing nonconvex
  cutting planes}, Journal of Global Optimization, 68 (2017), pp.~501--535.

\bibitem{joki2018double}
{\sc K.~Joki, A.~M. Bagirov, N.~Karmitsa, M.~M. M{\"a}kel{\"a}, and S.~Taheri},
  {\em Double bundle method for finding clarke stationary points in nonsmooth
  dc programming}, SIAM Journal on Optimization, 28 (2018), pp.~1892--1919.

\bibitem{Paper_Kojima2014}
{\sc M.~Kojima}, {\em Sums of squares relaxations of polynomial semidefinite
  programs}, in Mathematical and computing Sciences, Tokyo Institute of
  Technology, Meguro, Tokyo, 2014, pp.~152--8552.

\bibitem{SDPA}
{\sc M.~Kojima, K.~Fujisawa, K.~Nakata, and M.~Yamashita}, {\em Sdpa-m 7.3.9},
  \url{http://sdpa.sourceforge.net/}.

\bibitem{Paper_Lasserre2001}
{\sc J.~B. Lasserre}, {\em Global optimization with polynomials and the problem
  of moments}, Siam J Optim, 11 (2001), pp.~796--817.

\bibitem{Paper_Lasserre2008}
{\sc J.~B. Lasserre}, {\em A semidefinite programming approach to the
  generalized problem of moments.}, Mathematical Programming, 112 (2008),
  pp.~65--92.

\bibitem{Paper_Lasserre2009}
{\sc J.~B. Lasserre}, {\em Moments and sums of squares for polynomial
  optimization and related problems}, Journal of Global Optimization, 45
  (2009), pp.~39--61.

\bibitem{hoai2000efficient}
{\sc H.~A. Le~Thi}, {\em An efficient algorithm for globally minimizing a
  quadratic function under convex quadratic constraints}, Mathematical
  programming, 87 (2000), pp.~401--426.

\bibitem{le2014dc}
{\sc H.~A. Le~Thi, V.~N. Huynh, and D.~T. Pham}, {\em Dc programming and dca
  for general dc programs}, in Advanced Computational Methods for Knowledge
  Engineering: Proceedings of the 2nd International Conference on Computer
  Science, Applied Mathematics and Applications (ICCSAMA 2014), Springer, 2014,
  pp.~15--35.

\bibitem{Paper_Lethi2018}
{\sc H.~A. Le~Thi and D.~T. Pham}, {\em Dc programming and dca: thirty years of
  developments}, Math. Program., Special Issue dedicated to : DC Programming -
  Theory, Algorithms and Applications, 169 (2018), pp.~5--68.

\bibitem{lipp2016variations}
{\sc T.~Lipp and S.~Boyd}, {\em Variations and extension of the convex--concave
  procedure}, Optimization and Engineering, 17 (2016), pp.~263--287.

\bibitem{Yalmip}
{\sc J.~L\"{o}fberg}, {\em Yalmip : a toolbox for modeling and optimization in
  matlab}, Optimization, 2004 (2004), pp.~284 -- 289.

\bibitem{Mosek}
{\sc Mosek}, {\em Mosek modeling cookbook}, \url{https://www.mosek.com/}.

\bibitem{Paper_Murty1987}
{\sc K.~G. Murty and S.~N. Kabadi}, {\em Some np-complete problems in quadratic
  and nonlinear programming}, Mathematical Programming, 39 (1987),
  pp.~117--129.

\bibitem{Polylab}
{\sc Y.-S. Niu}, {\em Polylab -- a matlab multivariate polynomial toolbox}.
\newblock \url{https://github.com/niuyishuai/Polylab}.

\bibitem{Thesis_Niu2010}
{\sc Y.-S. Niu}, {\em Programmation DC \& DCA en Optimisation Combinatoire et
  Optimisation Polynomiale via les Techniques de SDP}, Minist\`{e}re de
  l'enseignement sup\'{e}rieur et de la recherche, Institut National Des
  Sciences Appliqu\'{e}es de Rouen, France, 2010.

\bibitem{niu2019discrete}
{\sc Y.-S. Niu and R.~Glowinski}, {\em Discrete dynamical system approaches for
  boolean polynomial optimization}, Journal of Scientific Computing, 92 (2022),
  pp.~1--39.

\bibitem{Paper_Niu2015}
{\sc Y.-S. Niu, J.~J. J\'{u}dice, H.~A. Le~Thi, and D.~T. Pham}, {\em Solving
  the quadratic eigenvalue complementarity problem by dc programming},
  Modelling, Computation and Optimization in Information Systems and Management
  Sciences, Advances in Intelligent Systems and Computing, 359 (2015),
  pp.~203--214.

\bibitem{niu2019improved}
{\sc Y.-S. Niu, J.~J. J\'{u}dice, H.~A. Le~Thi, and D.~T. Pham}, {\em Improved
  dc programming approaches for solving the quadratic eigenvalue
  complementarity problem}, Applied Mathematics and Computation, 353 (2019),
  pp.~95--113.

\bibitem{Paper_Niu2013}
{\sc Y.-S. Niu, H.~A. Le~Thi, D.~T. Pham, and J.~J. J\'{u}dice}, {\em Efficient
  dc programming approaches for the asymmetric eigenvalue complementarity
  problem}, Optimization Methods and Software, 28 (2013), pp.~812--829.

\bibitem{Paper_Niu2019higher}
{\sc Y.-S. Niu, Y.-J. Wang, H.~A. Le~Thi, and D.~T. Pham}, {\em High-order
  moment portfolio optimization via an accelerated difference-of-convex
  programming approach and sums-of-squares}, arXiv preprint arXiv:1906.01509,
  (2019).

\bibitem{pang2017computing}
{\sc J.-S. Pang, M.~Razaviyayn, and A.~Alvarado}, {\em Computing b-stationary
  points of nonsmooth dc programs}, Mathematics of Operations Research, 42
  (2017), pp.~95--118.

\bibitem{Paper_Pardalos1992}
{\sc P.~M. Pardalos and S.~A. Vavasis}, {\em Open questions in complexity
  theory for numerical optimization}, Springer-Verlag New York, Inc., 1992.

\bibitem{Paper_Parrilo2003}
{\sc P.~A. Parrilo}, {\em Semidefinite programming relaxations for
  semialgebraic problems}, Mathematical Programming, 96 (2003), pp.~293--320.

\bibitem{Paper_Pham1997}
{\sc D.~T. Pham and H.~A. Le~Thi}, {\em Convex analysis approach to d.c.
  programming: theory, algorithms and applications}, Acta Math. Vietnam., 22
  (1997), pp.~289--355.

\bibitem{Paper_Pham1998}
{\sc D.~T. Pham and H.~A. Le~Thi}, {\em Dc optimization algorithms for solving
  the trust region subproblem}, SIAM J. Optim., 8 (1998), pp.~476--507.

\bibitem{Paper_Pham2005}
{\sc D.~T. Pham and H.~A. Le~Thi}, {\em The dc programming and dca revisited
  with dc models of real world nonconvex optimization problems}, Annals of
  Operations Research, 133 (2005), pp.~23--46.

\bibitem{Paper_Niu2011}
{\sc D.~T. Pham and Y.-S. Niu}, {\em An efficient dc programming approach for
  portfolio decision with higher moments}, Computational Optimization \&
  Applications, 50 (2011), pp.~525--554.

\bibitem{nhat2018accelerated}
{\sc D.~N. Phan, H.~M. Le, and H.~A. Le~Thi}, {\em Accelerated difference of
  convex functions algorithm and its application to sparse binary logistic
  regression.}, in IJCAI, 2018, pp.~1369--1375.

\bibitem{SOSTOOLS}
{\sc S.~Prajna, A.~Papachristodoulou, and P.~A. Parrilo}, {\em Sostools: sum of
  squares optimization toolbox for matlab}, 2002–-2005,
  \url{http://www.mit.edu/~parrilo/sostools}.

\bibitem{Paper_Reznick1996}
{\sc B.~Reznick}, {\em Some concrete aspects of hilbert's 17th problem},
  Contemporary Mathematics,  (1996), pp.~251--272.

\bibitem{saunderson2022convex}
{\sc J.~Saunderson}, {\em A convex form that is not a sum of squares},
  Mathematics of Operations Research,  (2022).

\bibitem{shen2016disciplined}
{\sc X.~Shen, S.~Diamond, Y.~Gu, and S.~Boyd}, {\em Disciplined convex-concave
  programming}, in 2016 IEEE 55th Conference on Decision and Control (CDC),
  IEEE, 2016, pp.~1009--1014.

\bibitem{souza2016global}
{\sc J.~C.~O. Souza, P.~R. Oliveira, and A.~Soubeyran}, {\em Global convergence
  of a proximal linearized algorithm for difference of convex functions},
  Optimization Letters, 10 (2016), pp.~1529--1539.

\bibitem{SeDuMi}
{\sc J.~F. Sturm}, {\em Sedumi 1.3}, \url{http://sedumi.ie.lehigh.edu/}.

\bibitem{SDPT3}
{\sc K.~C. Toh, M.~J. Todd, and R.~H. Tutuncu}, {\em Sdpt3 - a matlab software
  package for semidefinite programming}, Optimization Methods \& Software, 11
  (1999), pp.~545--581.

\bibitem{Book_Tuy2016}
{\sc H.~Tuy}, {\em Convex Analysis and Global Optimization}, Springer, 2016.

\bibitem{wen2018proximal}
{\sc B.~Wen, X.~Chen, and T.~K. Pong}, {\em A proximal difference-of-convex
  algorithm with extrapolation}, Computational optimization and applications,
  69 (2018), pp.~297--324.

\bibitem{zou2006sparse}
{\sc H.~Zou, T.~Hastie, and R.~Tibshirani}, {\em Sparse principal component
  analysis}, Journal of computational and graphical statistics, 15 (2006),
  pp.~265--286.

\end{thebibliography}
\end{document}